\documentclass[preprint,1p,times,sort&compress]{elsarticle}
\usepackage{amsmath,amsfonts,mathrsfs}
\usepackage{amssymb}
\usepackage{color}
\usepackage{mathrsfs}
\usepackage{mathptmx}
\usepackage{verbatim}
\usepackage{epsfig}
\usepackage{CJK}
\usepackage{bm}
\usepackage{amsfonts}
\usepackage{amsthm}

\usepackage{amsthm}
\newtheorem{theorem}{Theorem}[section]
\newtheorem{lemma}{Lemma}[section]

\newdefinition{rem}{Remark}

\numberwithin{equation}{section}

\def\mi{{\bf i}}
\def\e{\varepsilon}

\begin{document}
\begin{frontmatter}
\title{Reducibility, Lyapunov exponent, pure point spectra property for quasi-periodic wave operator}

\author[SDUJ,SDUW]{Jing Li
}
 \ead{xlijing@sdu.edu.cn}
\address[SDUJ]{School of Mathematics , Shandong University, Jinan 250100, P.R. China}
\address[SDUW]{School of Mathematics and Statistics, Shandong University, Weihai 264209, P.R. China}

\begin{abstract}
In the present paper, the reducibility is derived for linear wave equation of finite smooth and time-quasi-periodic potential
 subject to Dirichlet boundary condition. Moreover, it is proved that the corresponding wave operator possesses the property of pure point spectra and zero Lyapunov exponent.
\end{abstract}

\end{frontmatter}

\section{Introduction}

  If a self-adjoint differential operator with time-quasi-periodic coefficients can be reduced to one with constant coefficients, the spectrum property and  Lyapunov exponent of the operator can be easily obtained. To this end, there are many literatures to deal with Schr\"odinger operator with time-quasi-periodic potential
of the form
 \begin{equation}\label{1}\mi\, \dot u=(H_0+\e W(\omega\, t,x,-\mi \nabla))u,\; x\in\mathbb R^d\;\;\text{or}\;\; x\in\mathbb T^d=\mathbb R^d/2\pi \mathbb Z^d,\end{equation}
 where $H_0=-\triangle +V(x)$ or an abstract self-adjoint (unbounded) operator and the perturbation $W$ is quasiperiodic in time $t$ and it may or may not depend on $x$ or/and $\nabla$.
 When $x\in\mathbb R^d$, there are many interesting and important results. See \cite{Bambusi17arxiv, Bambusi17CMP, Combescure1987Ann, Duclos-Stovicek96, Duclos-Stovicek02, Wang17nonlinearity}, and the references therein.
 When $x\in\mathbb T^d$ with any integer $d\ge 1$,  it is  in \cite{eliasson-kuksin},  proved that
\begin{equation}\label{2}\dot u=\mi (-\triangle+\e W(\phi_0+\omega \, t,x;\omega)u),\;\; x\in\mathbb T^d \end{equation}
is reduced to an autonomous equation
for most values of the frequency vector $\omega$, where $W$ is analytic in $(t,x)$ and quasiperiodic in time $t$ with frequency vector $\omega$. The reduction is made by means
of T\"oplitz-Lipschitz property of operator and very hard KAM technique. The basic difficult is in that the frequencies of the unperturbed operator $-\triangle$, denoted by $\lambda_k$ ($k\in\mathbb Z$), have multiplicity
 \[\lambda_k^{\sharp}\simeq |k|^{d-1}\to\infty,\quad \text{as}\;\; |k|\to\infty, \; \text{if}\;d>1.\]
 Fortunately, the frequencies have a good separate property and
 \[|\lambda_k-\lambda_{k^\prime}|\ge 1, \; \text{when}\; \lambda_k\neq \lambda_{k^\prime}.\]

\ \

When considering the reducibility for a linear wave equation

\begin{equation}\label{2w} u_{tt}=-\triangle+\e V(\phi_0+\omega t,x;\omega)u,\;\; x\in\mathbb T^d,\end{equation}
the unperturbed operator is $\sqrt{-\triangle}$ by writing \eqref{2w} as a system of order $1$. At this time, a  serious difficulty is that the frequencies of  $\sqrt{-\triangle}$  , still denoted by $\lambda_k$ ($k\in\mathbb Z$), have no good separate property and
 $|\lambda_k-\lambda_{k^\prime}|$ is dense, at least, in some interval of $\mathbb R$, when $d>1$. Thus, the reducibility for \eqref{2w} with $d>1$ is a challenging open problem. See \cite{Riccardo Montalto} for recent progress.
  However, the reducibility for \eqref{2w} with $d=1$ can be derived from the earlier KAM theorem ( \cite{Kuk1} and \cite{Poschel1996}) for nonlinear partial differential equations, assuming $V$ is analytic in $(t,x)$. Also see, \cite{Liang17JDDE}.
As we know, the spectrum property depends heavily on the smoothness of the perturbation for the discrete Schr\"odinger operator. For example, the Anderson localization and positivity of the Lyapunov exponent for one frequency discrete quasi-period Schr\"odinger operator with analytic potential occur in non-perturbative sense ( the largeness of the potential does not depend on the Diophantine condition. See \cite{Bourgain-Goldstin2000}, for the detail). However, one can only get perturbative results when the analytic property of  the potential is weaken to Gevrey regularity (see \cite{Klein2005}). Thus, a natural problem is that what happens when the perturbation $V$ is of finite smoothness in $(t,x)$.

\ \

Let us consider a linear wave equation with quasi-periodic coefficient:
\begin{equation}\label{eq1}
 u_{tt}-u_{xx}+\varepsilon V(\omega \, t,x)u=0
\end{equation}
subject to the boundary condition
\begin{equation}\label{eq2}
u(t,-\pi)=u(t,\pi)=0.
\end{equation}

{\bf Assumption A.} {\it Assume
$V$ is a $C^N$-smooth and  quasi-periodic in time $t$ with frequency $\omega\in \mathbb R^n$: that is, there is a hull function $\mathcal{V}(\theta,x)\in {C}^{N}(\mathbb{T}^{n}\times [-\pi, \pi], \mathbb{R})$
such that
$$ V(\omega\, t,x)=\left.\mathcal{V}(\theta, x)\right|_{\theta=\omega t},
\;\;\mathbb{T}^{n}=\mathbb{R}^{n}/2\pi \mathbb{Z}^{n}$$
where $N>200\, n$.}

\ \

{\bf Assumption B.} {\it We assume also that $V$ is  an even function of $x$, with zero-average:
$$\int_{-\pi}^{\pi}V(\omega\, t,x)dx\equiv 0.$$}

\ \

{\bf Assumption C.} {\it
Assume $\omega=\tau \omega_{0},$ where $\omega_{0}$ is Diophantine:

\begin{equation}\label{eq3}
|\langle k, \omega_{0} \rangle |\geq \frac{\gamma}{|k|^{n+1}},\;\;k\in \mathbb{Z}^{n}\setminus \{0\}
\end{equation}
with $0<\gamma\ll 1,$ a constant, and $\tau \in [1,2]$ is a parameter.}

\ \

Let $w=u_t$. Endow $L^2[-\pi,\pi]\times L^2[-\pi,\pi]$ a symplectic $ dw \wedge d u$.
Take  $(L^2[-\pi,\pi]\times L^2[-\pi,\pi]$, $dw \wedge d u)$ as phase space. Then \eqref{eq1} is a hamiltonian system with hamiltonian functional
\[H(u,w)=\int_{-\pi}^{\pi}\left(\frac12(w^2+ u_x^2)+\frac12 \e V(\omega t,x)u^2\right)\, d\, x,\]
and the Hamiltonian equation
\[u_t=\frac{\delta\, H}{\delta\, w},\quad w_t=-\frac{\delta\, H}{\delta u}.\]
\begin{theorem}\label{thm1.1} With the {\bf Assumptions A, B, C}, for given $1\gg\gamma>0$,
 there exists $\epsilon^*$ with $ \;0<\varepsilon^{*}=\varepsilon^{*} (n, \gamma)\ll \gamma,$
and exists a subset $\Pi\subset [1, 2]$ with
$$\mbox{Measure}\, \Pi\geq 1-O (\gamma^{1/3})$$
such that for any $0<\varepsilon<\varepsilon^{*}$ and for any $\tau \in \Pi,$ there is a quasi-periodic  symplectic change $u=\Phi(\omega \, t, x) v$ which changes
 \eqref{eq1} subject to \eqref{eq2}  into
\begin{equation}\label{yuan1.7} v_{tt}-v_{xx}+\varepsilon M_\xi v=0,\;\; v(t,-\pi)=v(t,\pi)=0,
\end{equation}
where $M_\xi$ is a real Fourier multiplier:
\[M_\xi \, \sin\, k x=\xi_k \sin\, k x,\; k\in \mathbb N,\]
with constants $\xi_k\in\mathbb R$ and $|\xi_k|\leq C /|k|$, where $C$ is absolute constant. Moreover, the wave operator
\[\mathcal L \, u(t,x)=(\partial_t^2-\partial_x^2+\varepsilon V(\omega\, t,x))\, u(t,x),\quad u(t,-\pi)=u(t,\pi)=0 \]
is of pure point spectrum property and of zero Lyapunov exponent.
\end{theorem}
\begin{rem} Noting that the operator $\mathcal L $ is self-adjoint in $L^2$ and using a technique in \cite{BG}, we also make the  coordinate change  $\Psi$ be unitary.
\end{rem}
\begin{rem} In \cite{Massimiliano-Berti}, it is proved that there is a quasi-periodic solution for $d$-dimensional nonlinear wave equation a quasi-periodic in time nonlinearity like
\[u_{tt}-\Delta u-V(x)\, u=\varepsilon f(\omega \, t, x, u),\;\; x\in \mathbb T^d,\]
where the multiplicative potential $V$ is in $C^q(\mathbb T^d;\, \mathbb R),\; \omega\in\mathbb R^n$ is a non-resonant frequency vector and $f\in C^q(\mathbb T^n\times \mathbb T^d\times\mathbb R;\, \mathbb R)$.
Because of the application of Nash-Moser iteration, it is not clear whether the obtained quasi-periodic solution is linear stable and has zero Lyapunov exponent. As a Corollary  of Theorem \ref{thm1.1},  we can prove that the quasi-periodic solution by \cite{Massimiliano-Berti} is linear stable and has zero Lyapunov exponent, when $d=1$.

\end{rem}

-----------------------------------------------------------------------------------

\section{Passing to Fourier coefficients}

Consider the differential equation:
\begin{equation}\label{eq4}
\mathcal{L}u=u_{tt}-u_{xx}+\varepsilon V(\omega t, x)u=0
\end{equation}
subject to the boundary condition

\begin{equation}\label{eq5}
u(t,-\pi)=u(t, \pi)=0.
\end{equation}
It is well-known that the Sturm-Liouville problem

\begin{equation}\label{eq6}
-y''=\lambda y,
\end{equation}
with the boundary condition

\begin{equation}\label{eq7}
y(-\pi)=y(\pi)=0
\end{equation}
has the eigenvalues and eigenfunctions, respectively,

\begin{equation}\label{eq8}
\lambda_{k}=k^{2},\;\;k=1, 2, \ldots,
\end{equation}
\begin{equation}\label{eq9}
\phi_{k}(x)=\sin kx,\;\;k=1, 2, \ldots .
\end{equation}
Make the ansatz
\begin{equation}\label{eq10}
u(t, x)=\mathcal{S} (u_k)=\sum_{k=1}^{\infty} u_{k}(t)\phi_{k}(x).
\end{equation}
Note that $V$ is an even function of $x$ such that
$\int_{-\pi}^{\pi}V(\omega\, t,x)dx\equiv 0.$ Write
\begin{equation}\label{eq11}
V(\omega t, x)=\sum_{k=1}^{\infty} v_{k}(\omega t)\varphi_{k}(x),
\end{equation}
where {$$\varphi_{k}(x)=\cos kx,\;\;k=1, 2, \ldots .$$}
Let
\begin{equation}\label{eq12}
\frac{du_{k}}{dt}=w_{k}.
\end{equation}
By the fact that
$$\varphi_{j}\phi_{l}=\sum_{k=1}^{\infty}\langle \varphi_{j}\phi_{l}, \phi_{k}\rangle \phi_{k},\;
j, l=1, 2, \cdots,$$
then \eqref{eq4} can be expressed as
\begin{equation*}
\sum_{k=1}^{\infty}\left(\frac{dw_{k}}{dt}+\lambda_{k}u_{k}+\varepsilon \sum_{l=1}^{\infty}\sum_{j=1}^{\infty}c_{jlk}v_{j}u_{l}\right)\phi_{k}=0,
\end{equation*}
which implies that
\begin{equation}\label{eq13}
\frac{dw_{k}}{dt}=-\lambda_{k}u_{k}-\varepsilon \sum_{l=1}^{\infty}\sum_{j=1}^{\infty}c_{jlk}v_{j}u_{l},
\end{equation}
where
\begin{equation}\label{eq2.11-}
c_{jlk}=\langle \varphi_{j}\phi_{l}, \phi_{k}\rangle
=\int_{-\pi}^{\pi}\cos jx\cdot\sin lx\cdot\sin kx\, dx=
\left\{%
\begin{array}{ll}
 \;\;0,\;k\neq \pm l\pm j,\\
\;\;\frac{\pi}{2},\; k= l\pm j \geq 1, \\
 -\frac{\pi}{2},\; k=-l\pm j\geq 1.\\
\end{array}%
\right.
\end{equation}
Rescale
\begin{equation}\label{eq14}
\mathcal{T}:\quad \left\{%
\begin{array}{cl}
w_{k}=\sqrt[4]{\lambda_{k}} p_{k},\\
\\
u_{k}=\frac{1}{\sqrt[4]{\lambda_{k}}}q_{k}.\\
\end{array}%
\right.
\end{equation}
Then

\begin{equation}\label{eq15}
\left\{%
\begin{array}{cl}
&\dot{q}_{k}=\sqrt{\lambda_{k}} p_{k},\\
\\
&\dot{p}_{k}=-\sqrt{\lambda_{k}}\,q_{k}-\varepsilon \sum_{l=1}^{\infty}\sum_{j=1}^{\infty}c_{jlk}\frac{v_{j}(\theta)}{\sqrt[4]{\lambda_{k}\lambda_{l}}}\,q_{l}.\\
\end{array}%
\right.
\end{equation}
This is a linear Hamiltonian system
\begin{equation}\label{eq16}
\left\{%
\begin{array}{cl}
&\dot{q}_{k}=\frac{\partial H}{\partial p_{k}},\\
\\
&\dot{p}_{k}=-\frac{\partial H}{\partial q_{k}},\\
\end{array}%
\right.
\end{equation}
where the symplectic structure is $d p\wedge d q=\sum_{j=1}^{\infty} d p_{j} \wedge d q_{j}$ and
the Hamiltonian function is
\begin{equation}\label{eq17}
H(p,q)=\sum_{k=1}^{\infty}\frac{\sqrt{\lambda_{k}}(p_{k}^{2}
+q_{k}^{2})}{2}+\varepsilon\sum_{k=1}^{\infty}\sum_{l=1}^{\infty}\sum_{j=1}^{\infty}
c_{jlk}\frac{v_{j}(\theta)}{\sqrt[4]{\lambda_{k}\lambda_{l}}}\,q_{l}q_{k},\quad \theta=\omega\, t.
\end{equation}
Introduce complex variables:
\begin{equation}\label{eq18}
z_{j}=\frac{1}{\sqrt{2}}(q_{j}-\sqrt{-1}p_{j}),\;\;\overline{z}_{j}=\frac{1}{\sqrt{2}}(q_{j}+\sqrt{-1}p_{j}),
\end{equation}
which is a symplectic transformation with $d p\wedge d q=\sqrt{-1} d z\wedge d \overline{z}$. Thus \eqref{eq16} is changed into
\begin{equation}\label{eq19}
\mathcal{G}:\quad \left\{%
\begin{array}{cl}
&\dot{z}_{k}=\sqrt{-1}\,\frac{\partial H}{\partial \overline{z}_{k}},\\
\\
&\dot{\overline{z}}_{k}=-\sqrt{-1}\,\frac{\partial H}{\partial z_{k}},\\
\end{array}%
\right.
\end{equation}
where
\begin{equation}\label{eq20}
H(z,\overline{z})=\sum_{k=1}^{\infty}\sqrt{\lambda_{k}}z_{k}
\overline{z}_{k}+\varepsilon\sum_{k=1}^{\infty}\sum_{l=1}^{\infty}\sum_{j=1}^{\infty}
c_{jlk}\frac{v_{j}(\theta)}{\sqrt[4]{\lambda_{k}\lambda_{l}}}
\left(\frac{z_{l}+\overline{z}_{l}}{\sqrt{2}}\right)\left(\frac{z_{k}+\overline{z}_{k}}{\sqrt{2}}\right).
\end{equation}
For two sequences $x=(x_{j}\in \mathbb{C},\;j=1, 2,\ldots)$, $y=(y_{j}\in \mathbb{C},\;j=1, 2,\ldots),$
define
$$\langle x, y\rangle=\sum_{j=1}^{\infty}x_{j}y_{j}.$$
Then we can write
\begin{equation}\label{eq21}
H=\langle\Lambda z,\overline{z}\rangle+\varepsilon \left[\langle R^{zz}(\theta)z, z\rangle
+\langle R^{z\overline{z}}(\theta)z, \overline{z}\rangle+\langle R^{\overline{z}\overline{z}}(\theta)\overline{z}, \overline{z}\rangle\right],
\end{equation}
where
\begin{equation*}
\Lambda=diag \left(\sqrt{\lambda_{j}}: j=1,2,\ldots\right),\;\;\theta=\omega t,
\end{equation*}
\begin{equation}\label{eq22.1}
R^{zz}(\theta)=\left(R^{zz}_{kl}(\theta): k, l=1,2,\ldots\right),
\;\; R^{zz}_{kl}(\theta)=\frac{1}{2}\sum_{j=1}^{\infty}\frac{c_{jlk}v_{j}(\theta)}{\sqrt[4]{\lambda_{k}}\sqrt[4]{\lambda_{l}}},
\end{equation}
\begin{equation}\label{eq22.2} R^{z\overline{z}}(\theta)=\left(R^{z\overline{z}}_{kl}(\theta): k, l=1,2,\ldots\right),
\;\; R^{z\overline{z}}_{kl}(\theta)=\sum_{j=1}^{\infty}\frac{c_{jlk}v_{j}(\theta)}{\sqrt[4]{\lambda_{k}}\sqrt[4]{\lambda_{l}}},
\end{equation}
\begin{equation}\label{eq22.3} R^{\overline{z}\,\overline{z}}(\theta)=\left(R^{\overline{z}\,\overline{z}}_{kl}(\theta): k, l=1,2,\ldots\right),
\;\; R^{\overline{z}\,\overline{z}}_{kl}(\theta)=\frac{1}{2}\sum_{j=1}^{\infty}
\frac{c_{jlk}v_{j}(\theta)}{\sqrt[4]{\lambda_{k}}\sqrt[4]{\lambda_{l}}}.
\end{equation}
Define a Hilbert space $h_{N}$ as follows:
\begin{equation}\label{eq23}
h_{N}=\{z=(z_{k}\in\mathbb{C}:k=1,2,\ldots)\}.
\end{equation}
Let $$\langle y, z\rangle_{N}:=\sum_{k=1}^{\infty}k^{2N}y_{k}\overline{z}_{k},\;\;\forall y, z\in h_{N}.$$
\begin{equation}\label{eq24}
\|z\|_{N}^{2}=\langle z, z\rangle_{N}.
\end{equation}
Recall that $$\mathcal{V}(\theta,x)\in {C}^{N}(\mathbb{T}^{n}\times [-\pi,\pi], \mathbb{R}).$$
Note that the Fourier transformation \eqref{eq10} is isometric from $u\in\mathcal{H}^{N}[-\pi, \pi]$
to $(u_{k}: k=1, 2, \ldots)\in h_{N},$ where $\mathcal{H}^{N}[-\pi, \pi]$ is the usual Sobolev space.
By \eqref{eq22.1}, \eqref{eq22.2} and \eqref{eq22.3}, we have that
\begin{eqnarray}\label{eq25}
&&\sup_{\theta\in\mathbb{T}^{n}}\|\sum_{|\alpha|= N}\partial^{\alpha}_{\theta}J R^{zz}(\theta) J\|_{h_N\to h_N}\leq C,\nonumber\\
&& \sup_{\theta\in\mathbb{T}^{n}}\|\sum_{|\alpha|= N}\partial^{\alpha}_{\theta}J R^{z\overline{z}}(\theta) J\|_{h_N\to h_N}\leq C,\\
&& \sup_{\theta\in\mathbb{T}^{n}}\|\sum_{|\alpha|= N}\partial^{\alpha}_{\theta}J R^{\overline{z}\,\overline{z}}(\theta) J\|_{h_N\to h_N}\leq C\nonumber,
\end{eqnarray}
where $ ||\cdot||_{h_N\to h_N}$ is the operator norm from $h_N$ to $h_N$, and $\alpha=(\alpha_{1}, \alpha_{2}, \ldots , \alpha_{n}),$ $|\alpha|=|\alpha_{1}|+|\alpha_{2}|+\ldots+|\alpha_{n}|,$ $\alpha_{j}$'s
 are positive integers, and $J=diag (\sqrt[4]{\lambda_{j}}: j=1, 2, \ldots)=diag (\sqrt{j}: j=1, 2, \ldots).$

Actually,
$$\partial_{\theta}^{\alpha}JR^{zz}(\theta)J=\left(\frac{1}{2}\sum_{j=1}^{\infty}
C_{jlk}\partial_{\theta}^{\alpha}v_{j}(\theta): l, k=1, 2, \cdots\right).$$
For any $z=(z_{k}\in \mathbb{C}: k=1, 2, \cdots)\in h_{N},$
$$\left(\sum_{|\alpha|=N}\partial_{\theta}^{\alpha}JR^{zz}(\theta)J\right)z=\left(\frac{1}{2}\sum_{j=1}^{\infty}
\sum_{k=1}^{\infty}C_{jlk}(\sum_{|\alpha|=N}\partial_{\theta}^{\alpha}v_{j}(\theta))z_{k}: l=1, 2, \cdots\right).$$
Let $$\gamma_{\,lj}=\frac{(\pm l\pm j)j}{l},\;\;\mbox{where}\;\;(\pm l\pm j)jl\neq 0.$$
Thus,
\begin{eqnarray*}\label{eq2.25}
&&\left\|\left(\sum_{|\alpha|=N}\partial_{\theta}^{\alpha}JR^{zz}(\theta)J\right)z\right\|_{N}^{2}\\
&=&\sum_{l=1}^{\infty}l^{2N}\left| \frac{1}{2}\sum_{j=1}^{\infty}
\sum_{k=1}^{\infty}C_{jlk}(\sum_{|\alpha|=N}\partial_{\theta}^{\alpha}v_{j}(\theta))z_{k}\right|^{2}\\
&=&\sum_{l=1}^{\infty}l^{2N}\left|\frac{1}{2}\sum_{j=1}^{\infty} C_{jl(\pm l\pm j)}(\sum_{|\alpha|=N}\partial_{\theta}^{\alpha}v_{j}(\theta))z_{\pm l\pm j}\right|^{2}\\
&=&\frac{1}{4}\sum_{l=1}^{\infty}l^{2N}\left|\frac{1}{\gamma_{\,lj}^{N}}\cdot \gamma_{\,lj}^{N}\sum_{j=1}^{\infty} C_{jl(\pm l\pm j)}(\sum_{|\alpha|=N}\partial_{\theta}^{\alpha}v_{j}(\theta))z_{\pm l\pm j}\right|^{2}\\
&\leq & C\left(\sum_{j=1}^{\infty}\frac{1}{\gamma_{\,lj}^{2N}}\right)\left(\sum_{j=1}^{\infty}|j|^{2N}\Big|\sum_{|\alpha|=N}\partial_{\theta}^{\alpha}v_{j}(\theta)
\Big|^{2}\sum_{l=1}^{\infty}|\pm l\pm j|^{2N}|z_{\pm l\pm j}|^{2}\right)\\
&\leq &C\sum_{j=1}^{\infty}|j|^{2N}\Big|\sum_{|\alpha|=N}\partial_{\theta}^{\alpha}v_{j}(\theta)\Big|^{2}\|z\|_{N}^{2}\\
&\leq &C\sup_{(\theta,x)\in \mathbb{T}^{n}\times [-\pi, \pi]}
\big|\sum_{|\alpha|=N}\partial_{\theta}^{\alpha}\partial_{x}^{N}\mathcal{V}(\theta, x)\big|\|z\|_{N}^{2}\leq C \|z\|_{N}^{2},
\end{eqnarray*}
where $C$ is a universal constant which might be different in different places.  It follows
\begin{equation}\label{eq2.26}
\sup_{\theta\in\mathbb{T}^{n}}\|\sum_{|\alpha|= N}\partial^{\alpha}_{\theta}J R^{z\overline{z}}(\theta) J\|_{h_N\to h_N}\leq C.
\end{equation}
The proofs of the last two inequalities in \eqref{eq25} are similar to that of \eqref{eq2.26}.
\section{Analytical Approximation Lemma}
We need to find a series of operators which are analytic in the some complex strip domains to approximate the operators
 $R^{zz}(\theta), R^{z\overline{z}}(\theta)$ and $R^{\overline{z}\,\overline{z}}(\theta)$. To this end, we cite an approximation lemma.
 See \cite{Salamon1989} and \cite{Salamon2004}. This method is used in \cite{yuan-zhang}, too.

We start by recalling some definitions and setting some new notations. Assume $X$ is a Banach space with the norm
$||\cdot||_{X}$. First recall that $C^{\mu}(\mathbb{R}^{n}; X)$ for $0< \mu <1$ denotes the space of bounded
H\"{o}lder continuous functions $f: \mathbb{R}^{n}\mapsto X$ with the form
$$\|f\|_{C^{\mu}, X}=\sup_{0<|x-y|<1}\frac{\|f(x)-f(y)\|_{X}}{|x-y|^{\mu}}+\sup_{x\in \mathbb{R}^{n}}\|f(x)\|_{X}.$$
If $\mu=0$ then $\|f\|_{C^{\mu},X}$ denotes the sup-norm. For $\ell=k+\mu$ with $k\in \mathbb{N}$ and $0\leq \mu <1,$
we denote by $C^{\ell}(\mathbb{R}^{n};X)$ the space of functions $f:\mathbb{R}^{n}\mapsto X$ with H\"{o}lder continuous partial derivatives, i.e., $\partial ^{\alpha}f\in C^{\mu}(\mathbb{R}^{n}; X_{\alpha})$ for all muti-indices $\alpha=(\alpha_{1}, \cdots, \alpha_{n})\in \mathbb{N}^{n}$ with the assumption that
$|\alpha|:=|\alpha_{1}|+\cdots+|\alpha_{n}|\leq k$ and $X_{\alpha}$ is the Banach space of bounded operators
$T:\prod^{|\alpha|}(\mathbb{R}^{n})\mapsto X$ with the norm
$$\|T\|_{X_{\alpha}}=\sup \{||T(u_{1}, u_{2}, \cdots, u_{|\alpha|})||_{X}:\|u_{i}\|=1, \;1\leq i \leq |\alpha|\}.$$
We define the norm
$$||f||_{C^{\ell}}=\sup_{|\alpha|\leq \ell}||\partial ^{\alpha}f||_{C^{\mu}, X_{\alpha}}$$
\begin{lemma}(Jackson)\label{Jackson}
Let $f\in C^{\ell}(\mathbb{R}^{n}; X)$ for some $\ell>0$ with finite $C^{\ell}$ norm over $\mathbb{R}^{n}.$
Let $\phi$ be a radical-symmetric, $C^{\infty}$ function, having as support the closure of the unit ball centered at the origin, where $\phi$ is completely flat and takes value 1, let $K=\widehat{\phi}$ be its Fourier transform. For all $\sigma >0$ define
$$ f_{\sigma}(x):=K_{\sigma}\ast f=\frac{1}{\sigma^{n}}\int_{\mathbb{R}^{n}}K(\frac{x-y}{\sigma})f(y)dy.$$
Then there exists a constant $C\geq 1$ depending only on $\ell$ and $n$ such that the following holds: For any $\sigma >0,$ the function $f_{\sigma}(x)$ is a real-analytic function from $\mathbb{C}^{n}/(\pi \mathbb{Z})^{n}$ to $X$ such that if $\Delta_{\sigma}^{n}$ denotes the $n$-dimensional complex strip of width $\sigma,$
$$\Delta_{\sigma}^{n}:=\{x\in \mathbb{C}^{n}\big ||\mathrm{Im} x_{j}|\leq \sigma,\;1\leq j\leq n\},$$
then for $\forall \alpha\in\mathbb{N}^{n}$ such that $|\alpha|\leq \ell$ one has

\begin{equation}\label{cite3.1}
\sup_{x\in \Delta_{\sigma}^{n}}||\partial ^{\alpha}f_{\sigma}(x)
-\sum_{|\beta|\leq \ell-|\alpha|}\frac{\partial^{\beta+\alpha}f(\mathrm{Re}x)}{\beta !}(\sqrt{-1}\mathrm{Im}x)^{\beta}||_{X_{\alpha}}\leq C ||f||_{C^{\ell}}\sigma^{\ell-|\alpha|},
\end{equation}

and for all $0\leq s\leq \sigma,$
\begin{equation}\label{cite3.2}
\sup_{x\in \Delta_{s}^{n}}\|\partial^{\alpha} f_{\sigma}(x)-\partial^{\alpha}f_{s}(x)\|_{X_{\alpha}}
\leq C ||f||_{C^{\ell}}\sigma^{\ell-|\alpha|}.
\end{equation}

The function $f_{\sigma}$ preserves periodicity (i.e., if $f$ is T-periodic in any of its variable $x_{j}$, so is $f_{\sigma}$). Finally, if $f$ depends on some parameter $\xi\in \Pi\subset\mathbb{R}^{n}$ and
 if $$||f(x,\xi)||_{C^{\ell}(X)}^{\mathcal{ L}}:=\sup_{\xi\in\Pi}||\partial_{\xi}\,f(x,\xi)|
|_{C^{\ell}(X)}$$ are uniformly bounded by a constant $C$
then all the above estimates hold with $\|\cdot\|$ replaced by $\|\cdot\|^{\mathcal{L}}.$
\end{lemma}
This lemma is very similar to the approximation theory obtained by Jackson, the only difference
is that we extend the applied range from $C^{\ell}(\mathbb{T}^{n}; \mathbb{C}^{n})$ to $C^{\ell}(\mathbb{T}^{n}; X).$ The proof of this lemma consists in a direct check which is based on standard tools from calculus and complex analysis. It is used to deal with KAM theory for finite smooth systems by Zehnder\cite{E. Zehnder}.Also see \cite{Chierchia} and \cite{yuan-zhang} and references therein, for example. For ease of notation, we shall replace $\|\cdot\|_{X}$ by
$\|\cdot\|.$ Now let us apply this lemma to the perturbation $P(\phi).$

Fix a sequence of fast decreasing numbers $s_{\nu}\downarrow 0, \upsilon\geq 0,$ and $s_{0}\leq \frac{1}{2}.$
For a $X$-valued function $P(\phi),$
construct a sequence of real analytic functions $P^{(\upsilon)}(\phi)$ such that the following conclusions holds:
\begin{description}
  \item[(1)] $P^{(\upsilon)}(\phi)$ is real analytic on the complex strip $\mathbb{T}^{n}_{s_{\upsilon}}$ of the width $s_{\upsilon}$ around $\mathbb{T}^{n}.$
  \item[(2)] The sequence of functions $P^{(\upsilon)}(\phi)$ satisfies the bounds:
  \begin{equation}\label{cite3.3}
  \sup_{\phi\in\mathbb{T}^{n}}\|P^{(\upsilon)}(\phi)-P(\phi)\|\leq C \|P\|_{C^{\ell}}s_{\upsilon}^{\ell},
  \end{equation}
  \begin{equation}\label{cite3.4}
  \sup_{\phi\in\mathbb{T}^{n}_{s_{\upsilon+1}}}\|P^{(\upsilon+1)}(\phi)-P^{(\upsilon)}(\phi)\|\leq C \|P\|_{C^{\ell}}s_{\upsilon}^{\ell},
  \end{equation}
  where $C$ denotes (different) constants depending only on $n$ and $\ell.$
  \item[(3)] The first approximate $P^{(0)}$ is ``small" with the perturbation $P$. Precisely speaking, for arbitrary $\phi\in\mathbb{T}^{n}_{s_{0}},$ we have
      \begin{eqnarray}\label{cite3.5}
      \|P^{(0)}(\phi)\|&\leq &\|P^{(0)}(\phi)-\sum_{|\alpha|\leq \ell}\frac{\partial^{\alpha}P(\mathrm{Re \phi})}{\alpha!}(\sqrt{-1}\mathrm{Im}\phi)^{\alpha}\|\nonumber\\
      &&+\|\sum_{|\alpha|\leq \ell}\frac{\partial^{\alpha}P(\mathrm{Re \phi})}{\alpha!}(\sqrt{-1}\mathrm{Im}\phi)^{\alpha}\|\nonumber\\
      &\leq & C(\|P\|_{C^{\ell}}s_{0}^{\ell}+\sum_{0\leq m\leq\ell}\|P\|_{C^{m}}s_{0}^{m})\nonumber\\
      &\leq & C\|P\|_{C^{\ell}}\sum_{m=0}^{\ell}s_{0}^{m}\nonumber\\
      &\leq & C\|P\|_{C^{\ell}}\sum_{m=0}^{\infty}s_{0}^{m}\nonumber\\
      &\leq & C\|P\|_{C^{\ell}},
      \end{eqnarray}
      where constant $C$ is independent of $s_{0},$ and the last inequality holds true due to the hypothesis that
      $s_{0}\leq \frac{1}{2}.$
  \item[(4)] From the first inequality \eqref{cite3.3}, we have the equality below. For arbitrary $\phi\in\mathbb{T}^{n},$
      \begin{equation}\label{cite3.6}
      P(\phi)=P^{(0)}(\phi)+\sum_{\upsilon=0}^{+\infty}(P^{(\upsilon+1)}(\phi)-P^{(\upsilon)}(\phi)).
      \end{equation}
\end{description}

Now take a sequence of real numbers $\{s_{v}\geq 0\}_{v=0}^{\infty}$ with $s_{v}>s_{v+1}$ goes fast to zero.
Let $R^{p,q}(\theta)=P(\theta)$ for $p, q\in \{z, \overline{z}\}.$ Then by \eqref{cite3.6}
 we can write, for $p, q\in\{z, \overline{z}\},$
  \begin{equation}\label{*}
 R^{p,q}(\theta)=R_{0}^{p,q}(\theta)+\sum_{l=1}^{\infty}R^{p,q}_{l}(\theta),
 \end{equation}
 where $R^{p,q}_{0}(\theta)$ is analytic in $\mathbb{T}_{s_{0}}^{n}$ with
  \begin{equation}\label{**}
 \sup_{\theta\in\mathbb{T}^{n}_{s_{0}}}\|R_{0}^{p,q}(\theta)\|_{h_{N}\to h_{N}}\leq C,
  \end{equation}
 and $R_{l}^{p,q}(\theta)\;(l\geq 1)$ is analytic in $\mathbb{T}^{n}_{s_{l}}$ with
 \begin{equation}\label{***}
 \sup_{\theta\in\mathbb{T}^{n}_{s_{l}}}\|J R_{l}^{p,q}(\theta) J\|_{h_{N}\to h_{N}}\leq Cs^{N}_{l-1}.
  \end{equation}
 \section{Iterative parameters of domains}
 Let

\begin{itemize}
  \item $\varepsilon_{0}=\varepsilon, \varepsilon_{\nu}=\varepsilon^{(\frac{4}{3})^{\nu}}, \nu=0, 1, 2, \ldots,$
  which measures the size of perturbation at $\nu-th$ step.
  \item $s_{\nu}=\varepsilon_{\nu+1}^{1/N}, \nu=0, 1, 2, \ldots,$
  which measures the strip-width of the analytic domain $\mathbb{T}_{s_{\nu}}^{n},$
  $\mathbb{T}_{s_{\nu}}^{n}=\{\theta\in \mathbb{C}^{n}/2\pi \mathbb{Z}^{n}: |Im \theta|\leq s_{\nu}\}.$
  \item $C({\nu})$ is a constant which may be different in different places, and it is of the form
   $$C({\nu})= C_{1} 2^{C_{2}\nu},$$
  where $C_{1},$ $C_{2}$ are  constants.
  \item $K_{\nu}=100 s_{\nu}^{-1}2^{\nu}|\log \varepsilon|.$
  \item$\gamma_{\nu}=\frac{\gamma}{2^{\nu}},\,0<\gamma\ll 1.$
  \item a family of subsets $\Pi_{\nu}\subset [1, 2]$ with $[1, 2]\supset \Pi_{0}\supset \ldots \supset \Pi_{\nu}\supset \ldots,$
  and $$mes \Pi_{\nu}\geq mes \Pi_{\nu-1}-C\gamma^{1/3}_{\nu-1}.$$

  \item For an operator-value (or a vector function) $B(\theta, \tau),$ whose domain is
  $(\theta, \tau)\in \mathbb{T}_{s_{\nu}}^{n}\times \Pi_{\nu}.$ Set
  $$ \|B\|_{\mathbb{T}^{n}_{s_{\nu}}\times \Pi_{\nu}}=\sup_{(\theta, \tau)
  \in \mathbb{T}^{n}_{s_{\nu}}\times \Pi_{\nu}}\|B(\theta, \tau)\|_{h_{N}\to h_{N}},$$ where $\|\cdot\|_{h_{N}\to h_{N}}$ is the operator norm, and set
  $$ \|B\|^{\mathcal{L}}_{\mathbb{T}^{n}_{s_{\nu}}\times \Pi_{\nu}}=\sup_{(\theta, \tau)
  \in \mathbb{T}^{n}_{s_{\nu}}\times \Pi_{\nu}}\|\partial_{\tau}B(\theta, \tau)\|_{h_{N}\to h_{N}}.$$
\end{itemize}
\section{Iterative Lemma}
In the following, for a function $f(\omega)$, denote by $\partial_{\omega}$ the derivative of $f(\omega)$ with respect to $\omega$ in Whitney's sense.
\begin{lemma}
For $p, q\in \{z, \overline{z} \},$ let $R_{0,0}^{p,q}=R_{0}^{p,q}, R_{l,0}^{p,q}=R_{l}^{p,q},$ where
$R_{0}^{p,q}, R_{l}^{p,q}$ are defined by \eqref{*}, \eqref{**} and \eqref{***}.
Assume that we have a family of Hamiltonian functions $H_{\nu}$:
\begin{equation}\label{eq26}
H_{\nu}=\sum_{j=1}^{\infty}\lambda_{j}^{(\nu)}z_{j}\overline{z}_{j}
+\sum_{l\geq \nu}^{\infty}\varepsilon_{l}(\langle R^{zz}_{l,\nu}z,z\rangle+\langle R^{z\overline{z}}_{l,\nu}z, \overline{z}\rangle
+\langle R^{\overline{z}\,\overline{z}}_{l,\nu}\overline{z}, \overline{z}\rangle),\;\nu=0,1, \ldots, m,
\end{equation}
where $R^{zz}_{l,\nu}, R^{z\overline{z}}_{l,\nu}, R^{\overline{z}\overline{z}}_{l,\nu}$ are operator-valued function defined on the domain $\mathbb{T}_{s_{\nu}}^{n}\times \Pi_{\nu},$ and
\begin{equation}\label{eq27}
\theta=\omega t.
\end{equation}
\begin{description}
  \item[$(A1)_{\nu}$]
  \begin{equation}\label{eq28}
  \lambda_{j}^{(0)}=\sqrt{\lambda_{j}}=j,\;
  \lambda_{j}^{(\nu)}=\sqrt{\lambda_{j}}+\sum_{i=0}^{\nu-1}\varepsilon_{i}\mu_{j}^{(i)},\;\;\nu\geq 1
  \end{equation}
and $\mu_{j}^{(i)}=\mu_{j}^{(i)}(\tau):\Pi_{i}\rightarrow\mathbb{R}$ with
\begin{eqnarray}
 &&|\mu_{j}^{(i)}|_{\Pi_{i}}:=\sup_{\tau\in\Pi_{i}}|\mu_{j}^{(i)}(\tau)|\leq C(i)/j,\;
  0\leq i\leq\nu-1,\label{eq29}\\ &&|\mu_{j}^{(i)}|_{\Pi_{i}}^{\mathcal{L}}:=\sup_{\tau\in\Pi_{i}}|\partial_{\tau}\mu_{j}^{(i)}(\tau)|\leq C(i)/j,\;
  0\leq i\leq\nu-1.\label{eq029}
  \end{eqnarray}
 \item[$(A2)_{\nu}$]
 For $p, q\in \{z,\overline{z}\},$ $R_{l,\nu}^{p,q}=R_{l,\nu}^{p,q}(\theta, \tau)$ is defined in
  $\mathbb{T}^{n}_{s_{l}}\times \Pi_{\nu}$ with $l\geq \nu,$ and is analytic in $\theta$ for fixed $\tau\in \Pi_{\nu},$
  and
  \begin{equation}\label{eq30}
  \|J R^{p,q}_{l,\nu} J\|_{\mathbb{T}^{n}_{s_{l}}\times \Pi_{\nu}}\leq C(\nu),
  \end{equation}
  \begin{equation}\label{eq31}
  \|J R^{p,q}_{l,\nu} J\|^{\mathcal{L}}_{\mathbb{T}^{n}_{s_{l}}\times \Pi_{\nu}}\leq C(\nu).
  \end{equation}
\end{description}
Then there exists a compact set $\Pi_{m+1}\subset \Pi_{m}$ with
 \begin{equation}\label{eq32}
 mes \Pi_{m+1}\geq mes \Pi_{m}-C\gamma_{m}^{1/3},
  \end{equation}
  and exists a symplectic coordinate changes
  \begin{equation}\label{eq33}
\Psi_{m}: \mathbb{T}^{n}_{s_{m+1}}\times \Pi_{m+1}\rightarrow \mathbb{T}^{n}_{s_{m}}\times \Pi_{m},
  \end{equation}
  \begin{equation}\label{eq33H}
||\Psi_{m}-id ||_{h_{N}\to h_{N}}\leq \varepsilon^{1/2},\;(\theta, \tau)\in \mathbb{T}^{n}_{s_{m+1}}\times \Pi_{m+1}
  \end{equation}
  such that the Hamiltonian function $H_{m}$ is changed into
 \begin{eqnarray}\label{eq34}
H_{m+1}& \triangleq & H_{m}\circ \Psi_{m}\nonumber\\
&=&\sum_{j=1}^{\infty}\lambda_{j}^{(m+1)}z_{j}\overline{z}_{j}
+\sum_{l\geq m+1}^{\infty}\varepsilon_{l}\left [\langle R^{zz}_{l,m+1}z,z\rangle\right.\\
&&\left.+\langle R^{z\overline{z}}_{l,m+1}z, \overline{z}\rangle
+\langle R^{\overline{z}\,\overline{z}}_{l,m+1}\overline{z}, \overline{z}\rangle\right ]\nonumber,
  \end{eqnarray}
  which is defined on the domain $\mathbb{T}^{n}_{s_{m+1}}\times \Pi_{m+1},$
  and ${\lambda_{j}^{(m+1)}}^{,}s$ satisfy the assumptions $(A 1)_{m+1}$ and
  $R_{l, m+1}^{p,q} (p,q\in \{z, \overline{z}\})$ satisfy the assumptions $(A 2)_{m+1}.$
\end{lemma}
\section{Derivation of homological equations}
Our end is to find a symplectic transformation $\Psi_{\nu}$ such that the terms ${R^{zz}_{l, v}}$, ${R^{z\overline{z}}_{l, v}}$, ${R^{\overline{z}\overline{z}}_{l, v}}$
(with $l=v$) disappear. To this end, let $F$ be a linear Hamiltonian of the form
 \begin{equation}\label{eq35}
F=\langle F^{zz}(\theta, \tau)z, z\rangle+\langle F^{z \overline{z}}(\theta, \tau)z, \overline{z}\rangle+\langle F^{\overline{z}\,\overline{z}}(\theta, \tau)\overline{z}, \overline{z}\rangle,
  \end{equation}
where $\theta=\omega t,$ $(F^{zz}(\theta, \tau))^{T}=F^{zz}(\theta, \tau),$ $(F^{z\overline{z}}(\theta, \tau))^{T}=F^{z\overline{z}}(\theta, \tau),$ $(F^{\overline{z}\overline{z}}(\theta, \tau))^{T}=F^{\overline{z}\overline{z}}(\theta, \tau).$
Moreover, let
\begin{equation}\label{eq36}
\Psi=\Psi_{m}=X_{\varepsilon_{m}F}^{t}\big | _{t=1},
 \end{equation}
 where $X^{t}_{\varepsilon_{m}F}$ is the flow of the Hamiltonian, $X_{\varepsilon_{m}F}$ is the vector field
 of the Hamiltonian $\varepsilon_{m}F$ with the symplectic structure $\sqrt{-1} dz\wedge d\overline{z}.$
 Let
 \begin{equation}\label{eq37}
H_{m+1}=H_{m}\circ \Psi_{m}.
 \end{equation}
 By \eqref{eq26}, we write
 \begin{equation}\label{eq38}
H_{m}=N_{m}+R_{m},
 \end{equation}
 with
\begin{equation}\label{eq39}
N_{m}=\sum_{j=1}^{\infty}\lambda_{j}^{(m)}z_{j}\overline{z}_{j},
 \end{equation}
 \begin{equation}\label{eq40}
R_{m}=\sum_{l=m}^{\infty}\varepsilon_{l}R_{lm},
 \end{equation}
 \begin{equation}\label{eq41}
R_{lm}=\langle R_{l,m}^{zz}(\theta)z, z\rangle+\langle R_{l,m}^{z\overline{z}}(\theta)z, \overline{z}\rangle+\langle R_{l,m}^{\overline{z}\,\overline{z}}(\theta)\overline{z}, \overline{z}\rangle,
 \end{equation}
 where $(R_{l,m}^{zz}(\theta))^{T}=R_{l,m}^{zz}(\theta), $ $(R_{l,m}^{z\overline{z}}(\theta))^{T}=R_{l,m}^{z\overline{z}}(\theta), $ $(R_{l,m}^{\overline{z} \overline{z}}(\theta))^{T}=R_{l,m}^{\overline{z}\overline{z}}(\theta).$
 Since the Hamiltonian $H_{m}=H_{m}(\omega t, z, \overline{z})$ depends on time $t,$ we introduce a fictitious
 action $I=$ constant, and let $\theta=\omega t$ be angle variable. Then the non-autonomous $H_{m}(\omega t, z, \overline{z})$ can be written as
 $$\omega I+H_{m}(\theta, z, \overline{z})$$
 with symplectic structure $d I\wedge d\theta +\sqrt{-1} d z\wedge d \overline{z}$.
 By combination of \eqref{eq35}-\eqref{eq41} and Taylor formula, we have
 \begin{eqnarray}\label{eq42}
H_{m+1}&=&H_{m}\circ X^{1}_{\varepsilon_{m}F}
\nonumber\\
&=&N_{m}+\varepsilon_{m}\{N_{m},F\}
+\varepsilon^{2}_{m}\int_{0}^{1}(1-\tau)\{\{N_{m}, F\}, F\}\circ X^{\tau}_{\varepsilon_{m}F}d\tau
+\varepsilon_{m}\omega\cdot\partial_{\theta}F\nonumber\\
&&+\varepsilon_{m}R_{mm}+(\sum_{l=m+1}^{\infty}\varepsilon_{l}R_{lm})\circ X^{1}_{\varepsilon_{m}F}
+\varepsilon_{m}^{2}\int_{0}^{1}\{R_{mm}, F\}\circ X^{\tau}_{\varepsilon_{m} F}d\tau,
 \end{eqnarray}
 where $\{\cdot, \cdot\}$ is the Poisson bracket with respect to $\sqrt{-1} dz\wedge d \overline{z},$
  that is $$\{H(z, \overline{z}), F(z, \overline{z})\}=
  \sqrt{-1}\left(\frac{\partial H}{\partial z}\cdot\frac{\partial F}{\partial \overline{z}}
  -\frac{\partial H}{\partial \overline{z}}\cdot\frac{\partial F}{\partial z}\right).$$
 Let $\Gamma_{K_{m}}$ be a truncation operator.
For any
$$
 f(\theta)=\sum_{k\in \mathbb{Z}^{n}}\widehat{f}(k)e^{i\langle k, \theta\rangle},\;\theta\in\mathbb{T}^{n}.
$$
 Define, for given $K_{m}>0,$
 $$\Gamma_{K_{m}} f(\theta)=(\Gamma_{K_{m}} f)(\theta)\triangleq \sum_{|k|\leq K_{m}}\widehat{f}(k)e^{i\langle k, \theta\rangle},$$
 $$(1-\Gamma_{K_{m}}) f(\theta)=((1-\Gamma_{K_{m}}) f)(\theta)\triangleq \sum_{|k|> K_{m}}\widehat{f}(k)e^{i\langle k, \theta\rangle}.$$
 Then
 $$f(\theta)=\Gamma_{K_{m}} f(\theta)+(1-\Gamma_{K_{m}})f(\theta).$$
 Let
 \begin{equation}\label{eq43}
\omega\cdot\partial_{\theta}F+\{N_{m}, F\}+\Gamma_{K_{m}}R_{mm}=\langle[R^{z\overline{z}}_{mm}]z, \overline{z}\rangle,
  \end{equation}
  where
 \begin{equation}\label{eq44}
  [R^{z\overline{z}}_{mm}]:=diag \left(\widehat{R}^{z\overline{z}}_{mmjj}(0): j=1, 2, \ldots\right),
  \end{equation}
  and $R^{z\overline{z}}_{mmij}(\theta)$ is the matrix element of $R^{z\overline{z}}_{mm}$ and $\widehat{R}^{z\overline{z}}_{mmij}(k)$ is the
  $k$-Fourier coefficient of $R^{z\overline{z}}_{mmij}(\theta).$
  Then
  \begin{equation}\label{eq45}
  H_{m+1}=N_{m+1}+C_{m+1}R_{m+1},
  \end{equation}
  where
  \begin{equation}\label{eq46}
 N_{m+1}=N_{m}+\varepsilon_{m}\langle [R^{z\overline{z}}_{mm}]z, \overline{z}\rangle
 =\sum_{j=1}^{\infty}\lambda_{j}^{(m+1)}z_{j}\overline{z}_{j},
  \end{equation}
\begin{equation}\label{eq5.23}
\lambda_{j}^{(m+1)}=\lambda_{j}^{(m)}+\varepsilon_{m}\widehat{R}_{mmjj}^{z\overline{z}}(0)
=\sqrt{\lambda_{j}}+\sum_{l=1}^{m}\varepsilon_{l}\mu_{j}^{(l)},
\;\mu_{j}^{(m)}:=\widehat{R}_{mmjj}^{z\overline{z}}(0).
  \end{equation}
  \begin{eqnarray}
 C_{m+1}R_{m+1}&=&\varepsilon_{m}(1-\Gamma_{K_{m}})R_{mm}\label{eq5.24}\\
 &+&\varepsilon_{m}^{2}\int_{0}^{1}(1-\tau)\{\{N_{m}, F\},F\}\circ X_{\varepsilon_{m}F}^{\tau}d\tau\label{eq5.27}\\
 &+&\varepsilon_{m}^{2}\int_{0}^{1}\{R_{mm}, F\}\circ X^{\tau}_{\varepsilon_{m} F}d\tau\label{eq5.28}\\
  &+&\left(\sum_{l=m+1}^{\infty}\varepsilon_{l}R_{lm}\right)\circ X_{\varepsilon_{m}F}^{1}.\label{eq5.25}
  \end{eqnarray}
  The equation \eqref{eq43} is called homological equation. Developing the Poisson bracket $\{N_{m},F\}$ and comparing
  the coefficients of $z_{i}{z}_{j}, z_{i}\overline{z}_{j}, \overline{z}_{i}\overline{z}_{j} (i, j=1, 2, \ldots),$ we get
  \begin{equation}\label{eq48}
 \left\{
    \begin{array}{ll}
\omega\cdot\partial_{\theta} F^{zz}(\theta, \tau)
      +\sqrt{-1}(\Lambda^{(m)}F^{zz}(\theta, \tau)+F^{zz}(\theta, \tau)\Lambda^{(m)})=\Gamma_{K_{m}}R_{mm}^{zz}(\theta),\\
\omega\cdot\partial _{\theta} F^{\overline{z}\,\overline{z}}(\theta, \tau)-\sqrt{-1}
     (\Lambda^{(m)}F^{\overline{z}\,\overline{z}}(\theta, \tau)+F^{\overline{z}\,\overline{z}}(\theta, \tau)\Lambda^{(m)})
=\Gamma_{K_{m}}R_{mm}^{\overline{z}\,\overline{z}}(\theta),  \\
\omega\cdot\partial_{\theta} F^{z\overline{z}}(\theta, \tau)
       +\sqrt{-1}(F^{z\overline{z}}(\theta, \tau)\Lambda^{(m)}-\Lambda^{(m)}F^{z\overline{z}}(\theta, \tau))
=\Gamma_{K_{m}}R_{mm}^{z\overline{z}}(\theta)-[R_{mm}],
    \end{array}
  \right.
\end{equation}
where
\begin{equation}\label{eq49}
\Lambda^{(m)}=diag(\lambda_{j}^{(m)}: j=1, 2, \ldots),
\end{equation}
and we assume
$$\Gamma_{K_{m}}F^{zz}(\theta, \tau)=F^{zz}(\theta, \tau), \Gamma_{K_{m}}F^{z\overline{z}}(\theta, \tau)=F^{z\overline{z}}(\theta, \tau), \Gamma_{K_{m}}F^{\overline{z}\, \overline{z}}(\theta, \tau)=F^{\overline{z}\, \overline{z}}(\theta, \tau).$$
Write $F^{zz}_{ij}(\theta), F^{z\overline{z}}_{ij}(\theta), F^{\overline{z} \,\overline{z}}_{ij}(\theta)$ are the matrix elements of $F^{zz}(\theta, \tau), F^{z\overline{z}}(\theta, \tau), F^{\overline{z} \,\overline{z}}(\theta, \tau),$ respectively.
Then \eqref{eq48} can be rewritten as:
\begin{equation}\label{eq5.31}
\omega\cdot\partial_{\theta} F_{ij}^{zz}(\theta)
       +\sqrt{-1}(\lambda_{i}^{(m)}+\lambda_{j}^{(m)})F_{ij}^{zz}(\theta)=
       \Gamma_{K_{m}}R_{mmij}^{zz}(\theta),
\end{equation}
\begin{equation}\label{eq5.32}
     \omega\cdot\partial _{\theta} F_{ij}^{\overline{z}\,\overline{z}}(\theta)-\sqrt{-1}(\lambda_{i}^{(m)}
     +\lambda_{j}^{(m)})F_{ij}^{\overline{z}\,\overline{z}}(\theta)
=\Gamma_{K_{m}}R_{mmij}^{\overline{z}\,\overline{z}}(\theta),
\end{equation}

    \begin{equation}\label{eq5.33}
\omega\cdot\partial_{\theta} F_{ij}^{z\overline{z}}(\theta)
       -\sqrt{-1}(\lambda_{i}^{(m)}-\lambda_{j}^{(m)})F_{ij}^{z\overline{z}}(\theta)=
       \Gamma_{K_{m}}R_{mmij}^{z\overline{z}}(\theta),\;
i\neq j,
\end{equation}
 \begin{equation}\label{eq5.34}
\omega\cdot\partial_{\theta} F_{ii}^{z\overline{z}}(\theta)=\Gamma_{K_{m}}R_{mmii}^{z\overline{z}}(\theta)-\widehat{R}_{mmii}(0),
\end{equation}

where $i,j=1, 2, \ldots.$
\section{Solutions of the homological equations}
\begin{lemma}\label{lem7.1}
There exists a compact subset $\Pi_{m+1}^{+-}\subset \Pi_{m}$ with
\begin{equation}\label{eq7.1}
mes(\Pi_{m+1}^{+-})\geq mes \Pi_{m}-C\gamma_{m}^{1/3}
\end{equation}
such that for any $\tau\in \Pi_{m+1}^{+-}$ (Recall $\omega=\tau\omega_{0}$),
the equation \eqref{eq5.33} has a unique solution $F^{z\overline{z}}(\theta, \tau),$
which is defined on the domain $\mathbb{T}_{s_{m+1}}^{n}\times \Pi_{m+1}^{+-},$
with
\begin{equation}\label{eq7.2}
\|JF^{z\overline{z}}(\theta,\tau)J\|_{\mathbb{T}_{s_{m+1}}^{n}\times
\Pi_{m+1}^{+-}}\leq C(m+1)\varepsilon_{m}^{-\frac{6(n+1)}{N}},
\end{equation}
\begin{equation}\label{eq7.3}
\|JF^{z\overline{z}}(\theta,\tau)J\|^{\mathcal{L}}_{\mathbb{T}_{s_{m+1}}^{n}\times
\Pi_{m+1}^{+-}}\leq C(m+1)\varepsilon_{m}^{-\frac{12(n+1)}{N}}.
\end{equation}
\end{lemma}
\begin{proof}
By passing to Fourier coefficients, \eqref{eq5.33} can be rewritten as
\begin{equation}\label{eq7.4}
(-\langle k, \omega\rangle+\lambda_{i}^{(m)}-\lambda_{j}^{(m)})\widehat{F}_{ij}^{z\overline{z}}(k)
=\sqrt{-1}\widehat{R}_{mmij}^{z\overline{z}}(k),
\end{equation}
where $i, j=1, 2, \ldots, k\in \mathbb{Z}^{n}$ with $|k|\leq K_{m}.$ Recall $\omega=\tau \omega_{0}.$

Let
$$A_{k}=|k|^{2n+3}+8,$$
and let
\begin{equation}\label{eq7.5}
Q_{kij}^{(m)}\triangleq \left\{\tau\in\Pi_{m}\bigg| \big|-\langle k, \omega_{0}\rangle\tau
+\lambda_{i}^{(m)}-\lambda_{j}^{(m)}\big|<\frac{(|i-j|+1)\gamma_{m}}{A_{k}}\right\},
\end{equation}
where $i, j=1, 2, \ldots,$ $k\in\mathbb{Z}^{n}$ with $|k|\leq K_{m},$ and $k\neq 0$ when $i=j.$
Let
\begin{equation*}
\Pi_{m+1}^{+-}=\Pi_{m}\diagdown\bigcup_{|k|\leq K_{m}}\bigcup_{i=1}^{\infty}\bigcup_{j=1}^{\infty}Q_{kij}^{(m)}.
\end{equation*}
Then for any $\tau\in\Pi_{m+1}^{+-},$ we have
\begin{equation}\label{eq7.6}
\big|-\langle k, \omega\rangle
+\lambda_{i}^{(m)}-\lambda_{j}^{(m)}\big|\geq\frac{(|i-j|+1)\gamma_{m}}{A_{k}}.
\end{equation}
Recall that $R_{mm}^{{z}\overline{z}}(\theta)$ is analytic in the domain $\mathbb{T}_{s_{m}}^{n}$
for any $\tau\in \Pi_{m},$
\begin{equation}\label{eq7.8}
|\widehat{R}_{mmij}^{z\overline{{z}}}(k)|
\leq \frac{C(m)}{\sqrt{ij}}e^{-s_{m}|k|}.
\end{equation}
It follows
\begin{eqnarray}\label{eq7.7}
|\widehat{F}_{ij}^{z\overline{z}}(k)|
&=&\Bigg|\frac{\widehat{R}_{mmij}^{z\overline{z}}(k)}{-\langle k, \omega\rangle +\lambda_{i}^{(m)}-\lambda_{j}^{(m)}}\Bigg|\leq \frac{A_{k}}{\gamma_{m}(|i-j|+1)}\cdot|\widehat{R}_{mmij}^{{z}\overline{z}}(k)|\nonumber\\
&\leq &
\frac{(|k|^{2n+3}+8)}{\gamma_{m}(|i-j|+1)}\cdot\frac{C(m)}{\sqrt{ij}}e^{-s_{m}|k|}.
\end{eqnarray}

Now we need the following lemmas:
\begin{lemma}\label{lem7.2}
For $0<\delta<1, \nu>1,$ one has
$$\sum_{k\in \mathbb{Z}^{n}}e^{-2|k|\delta}|k|^{\nu}<\left(\frac{\nu}{e}\right)^{\nu}\frac{(1+e)^{n}}{\delta^{\nu+n}}
.$$
\end{lemma}
\begin{proof}
The proof can be found in \cite{Bogoljubov1969}.
\end{proof}

\begin{lemma}\label{lem7.3}
If $A=(A_{ij})$ is a bounded linear operator on $\ell^{2},$ then also $B=(B_{ij})$ with
$$B_{ij}=\frac{|A_{ij}|}{|i-j|},\;i\neq j,$$
and $B_{ii}=0$ is bounded linear operator on $\ell^{2},$ and $\|B\|\leq \left(\frac{\pi}{\sqrt{3}}\right)\|A\|,$
where $\|\cdot\|$ is $\ell^{2}\rightarrow\ell^{2}$ operator norm.
\end{lemma}

\begin{proof}
See P\"{o}schel, a KAM theorem for PDE \cite{Poschel1996}.
\end{proof}
\begin{rem}\label{rem1}
Lemma \ref{lem7.3} holds true for the weight norm $\|\cdot\|_{N}.$
\end{rem}

Therefore, by \eqref{eq7.7}, we have
\begin{eqnarray*}
&&\sup_{\theta\in\mathbb{T}^{n}_{s'_{m}}\times \Pi_{m+1}}(\sqrt{i}|F^{z\overline{z}}_{ij}(\theta, \tau)|\sqrt{j})\\
&\leq &\left(\sum_{|k|\leq K_{m}}(|k|^{2n+3}+8)e^{-(s_{m}-s'_{m})|k|}\right)\cdot\frac{C(m)}{\gamma_{m}(|i-j|+1)}\\
&\leq & C\left(\frac{2n+3}{e}\right)^{2n+3}\!(1+e)^{n}\left(\frac{2}{s_{m}-s'_{m}}\right)^{3n+3}\!\cdot
\frac{C(m)}{\gamma_{m}(|i-j|+1)}
\;( \mbox{by Lemma 7.2})\\
&\leq & C\cdot\frac{C(m)}{(s_{m}-s'_{m})^{3(n+1)}}\cdot\frac{1}{\gamma_{m}(|i-j|+1)}\\
&\leq & C\cdot\varepsilon_{m}^{-\frac{6(n+1)}{N}}\cdot\frac{C(m)}{\gamma_{m}(|i-j|+1)},
\end{eqnarray*}
where $C$ is a constant depending on $n,$ $s'_{m}=s_{m}-\frac{s_{m}-s_{m+1}}{4}.$

By Lemma \ref{lem7.3} and the Remark \ref{rem1}, we have
\begin{equation}\label{eq7.9}
\|JF^{z\overline{z}}(\theta, \tau)J\|_{\mathbb{T}^{n}_{s'_{m}}\times \Pi^{+-}_{m+1}}
\leq C\cdot C(m)\gamma_{m}^{-1}\varepsilon_{m}^{-\frac{6(n+1)}{N}}
\leq C(m+1)\varepsilon_{m}^{-\frac{6(n+1)}{N}}.
\end{equation}
It follows $s'_{m}>s_{m+1}$ that
$$\|JF^{z\overline{z}}(\theta, \tau)J\|_{\mathbb{T}^{n}_{s_{m+1}}\times \Pi^{+-}_{m+1}}
\leq\|JF^{z\overline{z}}(\theta, \tau)J\|_{\mathbb{T}^{n}_{s'_{m}}\times \Pi^{+-}_{m+1}}
\leq C(m+1)\varepsilon_{m}^{-\frac{6(n+1)}{N}}.$$
Applying $\partial_{\tau}$ to both sides of \eqref{eq7.4}, we have
\begin{equation}\label{eq7.10}
\left(-\langle k, \omega\rangle +\lambda_{i}^{(m)}-\lambda_{j}^{(m)}\right)\partial_{\tau}\widehat{F}_{ij}^{z\overline{z}}(k)
=\sqrt{-1}\partial_{\tau}\widehat{R}_{mmij}^{z\overline{z}}(k)+(*),
\end{equation}
where
\begin{equation}\label{eq7.11}
(*)=-\left(-\langle k, \omega_{0}\rangle+\partial_{\tau}(\lambda_{i}^{(m)}-\lambda_{j}^{(m)})\right)\widehat{F}_{ij}^{z\overline{z}}(k).
\end{equation}
Recalling $|k|\leq K_{m}=100s_{m}^{-1}2^{m}|\log \varepsilon|,$ and using \eqref{eq28} and \eqref{eq29} with $\nu=m,$ and
using \eqref{eq7.9}, we have, on $\tau\in\Pi_{m+1},$
\begin{equation}\label{eq7.12}
\sqrt{i}|(*)|\sqrt{j}\leq C\cdot C(m)\gamma_{m}^{-1}\varepsilon_{m}^{-\frac{6(n+1)}{N}}e^{-s'_{m}|k|}.
\end{equation}
According to \eqref{eq31},
\begin{equation}\label{eq7.13}
\mid \sqrt{i} \partial _{\tau}\widehat{R}^{z\overline{z}}_{mmij}(k)\sqrt{j}\mid \leq C(m)e^{-s_{m}|k|}.
\end{equation}
By \eqref{eq7.6}, \eqref{eq7.10}, \eqref{eq7.12} and \eqref{eq7.13}, we have
\begin{equation}\label{eq7.14}
\mid \sqrt{i} \partial _{\tau}\widehat{F}^{z\overline{z}}_{ij}(k)\sqrt{j}\mid
\leq \frac{A_{k}}{\gamma_{m}(|i-j|+1)}\cdot C\cdot C(m)\gamma_{m}^{-1}\varepsilon_{m}^{-\frac{6(n+1)}{N}}e^{-s'_{m}|k|} \;\;\mbox{for}\;\; i\neq j.
\end{equation}
Note that $s_{m}>s'_{m}>s_{m+1}.$ Again using Lemma \ref{lem7.2} and Lemma\ref{lem7.3}, we have
\begin{eqnarray}\label{eq7.15}
&&\|JF^{z\overline{z}}(\theta, \tau)J\|^{\mathcal{L}}_{\mathbb{T}_{s_{m+1}}\times \Pi_{m+1}^{+-}}
=\| J \partial _{\tau}{F}^{z\overline{z}}(\theta, \tau)J\|_{\mathbb{T}_{s_{m+1}}\times \Pi_{m+1}^{+-}}
\nonumber\\
&\leq &C^{2}\cdot C(m)\gamma^{-2}_{m}\varepsilon_{m}^{-\frac{12(n+1)}{N}}\leq C(m+1)\varepsilon_{m}^{-\frac{12(n+1)}{N}}.
\end{eqnarray}
The proof of the measure estimate \eqref{eq7.1} will be postponed to section 10.
This completes the proof of Lemma \ref{lem7.1}.
\end{proof}

\begin{lemma}\label{lem7.4}
There exists a compact subset $\Pi_{m+1}^{++}\subset \Pi_{m}$ with
\begin{equation}\label{eq a}
mes(\Pi_{m+1}^{++})\geq (mes \Pi_{m}-C\gamma_{m}^{1/3})
\end{equation}
such that for any $\tau\in \Pi_{m+1}^{++}$ (Recall $\omega=\tau\omega_{0}$),
the equation \eqref{eq5.31} has a unique solution $F^{z{z}}(\theta),$
which is defined on the domain $\mathbb{T}_{s_{m+1}}^{n}\times \Pi_{m+1}^{++},$
with
\begin{equation*}
\|JF^{z{z}}(\theta,\tau)J\|_{\mathbb{T}_{s_{m+1}}^{n}\times
\Pi_{m+1}^{++}}\leq C(m+1)\varepsilon_{m}^{-\frac{6(n+1)}{N}},
\end{equation*}
\begin{equation*}
\|JF^{z{z}}(\theta,\tau)J\|^{\mathcal{L}}_{\mathbb{T}_{s_{m+1}}^{n}\times
\Pi_{m+1}^{++}}\leq C(m+1)\varepsilon_{m}^{-\frac{12(n+1)}{N}}.
\end{equation*}
\end{lemma}

\begin{lemma}\label{lem7.5}
There exists a compact subset $\Pi_{m+1}^{--}\subset \Pi_{m}$ with
\begin{equation}\label{eq b}
mes(\Pi_{m+1}^{--})\geq (mes \Pi_{m}-C\gamma_{m}^{1/3})
\end{equation}
such that for any $\tau\in \Pi_{m+1}^{--}$ (Recall $\omega=\tau\omega_{0}$),
the equation \eqref{eq5.32} has a unique solution $F^{\overline{z}\overline{z}}(\theta),$
which is defined on the domain $\mathbb{T}_{s_{m+1}}^{n}\times \Pi_{m+1}^{--},$
with
\begin{equation*}
\|JF^{\overline{z}\overline{z}}(\theta,\tau)J\|_{\mathbb{T}_{s_{m+1}}^{n}\times
\Pi_{m+1}^{--}}\leq C(m+1)\varepsilon_{m}^{-\frac{6(n+1)}{N}},
\end{equation*}
\begin{equation*}
\|JF^{\overline{z}\overline{z}}(\theta,\tau)J\|^{\mathcal{L}}_{\mathbb{T}_{s_{m+1}}^{n}\times
\Pi_{m+1}^{--}}\leq C(m+1)\varepsilon_{m}^{-\frac{12(n+1)}{N}}.
\end{equation*}
\end{lemma}
The proofs of Lemma \ref{lem7.4} and Lemma \ref{lem7.5} are a little bit simpler than that of Lemma \ref{lem7.1}.
So we omit them.

Let
$$\Pi_{m+1}=\Pi_{m+1}^{+-}\bigcap\Pi_{m+1}^{++}\bigcap\Pi_{m+1}^{--}.$$
By \eqref{eq7.1}, \eqref{eq a} and \eqref{eq b}, we have
$$mes \Pi_{m+1}\geq mes \Pi _{m}-C\gamma_{m}^{1/3}.$$
\section{Coordinate change $\Psi$ by $\varepsilon_{m} F$}
Recall $\Psi=\Psi_{m}=X_{\varepsilon_{m}F}^{t}\big| _{t=1},$ where $X^{t}_{\varepsilon_{m}F}$ is the flow of the Hamiltonian $\varepsilon_{m} F$, vector field $X_{\varepsilon_{m}F}$ with symplectic $\sqrt{-1}dz\wedge d\overline{z}.$
So
$$\sqrt{-1} \dot{z}=\varepsilon_{m}\frac{\partial F}{\partial \overline{z}},\;
-\sqrt{-1}\dot{\overline{z}}=\varepsilon_{m}\frac{\partial F}{\partial {z}},\;\dot{\theta}=\omega.$$
More exactly,
$$\left\{
    \begin{array}{ll}
      \sqrt{-1}\dot{z}=\varepsilon_{m}(F^{z\overline{z}}(\theta, \tau)z
+2F^{\overline{z}\,\overline{z}}(\theta, \tau)\overline{z}),\;\theta=\omega t,\\
     -\sqrt{-1}\dot{\overline{z}}=\varepsilon_{m}(2F^{{z}{z}}(\theta, \tau){z}+F^{z\overline{z}}(\theta, \tau)\overline{z}),\;\theta=\omega t, \\
 \dot{\theta}=\omega.
    \end{array}
  \right.
$$

Let $u=\left(
         \begin{array}{c}
           z \\
           \overline{z} \\
         \end{array}
       \right),$
\begin{equation}\label{eq8.0}
B_{m}=\left(
            \begin{array}{cc}
            -\sqrt{-1} F^{z\overline{z}}(\theta, \tau) & -2\sqrt{-1}F^{\overline{z}\,\overline{z}}(\theta, \tau) \\
             2\sqrt{-1} F^{zz}(\theta, \tau) & \sqrt{-1} F^{z\overline{z}}(\theta, \tau)\\
            \end{array}
          \right).\;\;\mbox{Recall }\;\;\theta=\omega t.
\end{equation}
Then
\begin{equation}\label{eq8.1}
\frac{du(t)}{dt}=\varepsilon_{m}B_{m}(\theta)u,\;\;\dot{\theta}=\omega.
\end{equation}
Let $u(0)=u_{0}\in h_{N}\times h_{N},\,\theta(0)=\theta_{0}\in \mathbb{T}^{n}_{s_{m+1}}$ be
initial value.
Then
\begin{equation}\label{eq8.2}
\left\{
  \begin{array}{ll}
    u(t)=u_{0}+\int_{0}^{t}\varepsilon_{m}B_{m}(\theta_{0}+\omega s)u(s)ds, \\
    \theta(t)=\theta_{0}+\omega t.
  \end{array}
\right.
\end{equation}
By Lemmas \ref{lem7.1}, \ref{lem7.4} and \ref{lem7.5} in Section 7,
\begin{equation}\label{eq8.3}
\|J B_{m}(\theta)J\|_{\mathbb{T}^{n}_{s_{m+1}}\times \Pi_{m+1}}\leq C(m+1)\varepsilon_{m}^{-\frac{6(n+1)}{N}},
\end{equation}
\begin{equation}\label{eq8.4}
\|J B_{m}(\theta) J\|^{\mathcal{L}}_{\mathbb{T}^{n}_{s_{m+1}}\times \Pi_{m+1}}\leq C(m+1)\varepsilon_{m}^{-\frac{12(n+1)}{N}}.
\end{equation}

It follows from \eqref{eq8.2} that
$$u(t)-u_{0}=\int_{0}^{t}\varepsilon_{m}B_{m}(\theta_{0}+\omega s)u_{0}ds+\int_{0}^{t}\varepsilon_{m}B_{m}(\theta_{0}+\omega s)(u(s)-u_{0})ds.$$
Moreover, for $t\in [0,1]$, $\|u_{0}\|_{N}\leq 1,$
\begin{equation}\label{eq8.5}
\|u(t)-u_{0}\|_{N}\leq \varepsilon_{m}C(m+1)\varepsilon_{m}^{-\frac{6(n+1)}{N}}
+\int_{0}^{t}\varepsilon_{m}\|B_{m}(\theta_{0}+\omega s)\|\|u(s)-u_{0}\|_{N}ds,
\end{equation}
where $\|\cdot\|$ is the operator norm from $h_{N}\times h_{N}\rightarrow h_{N}\times h_{N}.$

By Gronwall's inequality,
\begin{equation}\label{eq8.6}
\|u(t)-u_{0}\|_{N}\leq C(m+1)\varepsilon_{m}^{1-\frac{6(n+1)}{N}}
\cdot\exp\left(\int_{0}^{t}\varepsilon_{m}\|B_{m}(\theta_{0}+\omega s)\|ds\right)\leq \varepsilon_{m}^{1/2}.
\end{equation}
Thus,
\begin{equation}\label{eq8.7}
\Psi_{m}: \mathbb{T}^{n}_{s_{m+1}}\times \Pi_{m+1}\rightarrow \mathbb{T}^{n}_{s_{m}}\times \Pi_{m},
\end{equation}
and
\begin{equation}\label{eq8.8}
\|\Psi_{m}-id\|_{h_{N}\to h_{N}}\leq \varepsilon_{m}^{1/2}.
\end{equation}
Since \eqref{eq8.1} is linear, so $\Psi_{m}$ is linear coordinate change. According to \eqref{eq8.2}, construct
Picard sequence:
$$\left\{
    \begin{array}{ll}
      u_{0}(t)=u_{0}, \\
      u_{j+1}(t)=u_{0}+\int_{0}^{t}\varepsilon_{m}B(\theta_{0}+\omega s)u_{j}(s)ds,\;j=0, 1, 2,\ldots.
    \end{array}
  \right.
$$
By \eqref{eq8.8}, this sequence with $t=1$ goes to
\begin{equation}\label{eq8.10}
\Psi_{m}(u_{0})=u(1)=(id+P_{m}(\theta_{0}))u_{0},
\end{equation}
where $id$ is the identity from $h_{N}\times h_{N}\rightarrow h_{N}\times h_{N},$ and $P(\theta_{0})$ is an operator form $h_{N}\times h_{N}\rightarrow h_{N}\times h_{N}$
for any fixed $\theta_{0}\in \mathbb{T}^{n}_{{s_{m+1}}}, \tau\in\Pi_{m+1},$
and is analytic in $\theta_{0}\in \mathbb{T}^{n}_{{s_{m+1}}},$ with
\begin{equation}\label{eq8.11}
\|P_{m}(\theta_{0})\|_{\mathbb{T}^{n}_{s_{m+1}}\times \Pi_{m+1}}\leq \varepsilon_{m}^{1/2}.
\end{equation}
Note that \eqref{eq8.1} is a Hamiltonian system. So $P_{m}(\theta_{0})$ is a symplectic linear operator from
$h_{N}\times h_{N}$ to $h_{N}\times h_{N}.$
\section{Estimates of remainders}
The aim of this section is devoted to estimate the remainders:
$$C_{m+1}R_{m+1}=(6.14)+\ldots + (6.17).$$
\begin{itemize}
  \item Estimate of \eqref{eq5.24}.

By \eqref{eq41}, let
$$\widetilde{R}_{mm}=\widetilde{R}_{mm}(\theta)=\left(
            \begin{array}{cc}
            R_{_{m,m}}^{zz}(\theta) & \frac{1}{2}R_{m,m}^{z\overline{z}}(\theta) \\
              \frac{1}{2}R_{m,m}^{z\overline{z}}(\theta) & R_{m,m}^{\overline{z}\overline{z}}(\theta)\\
            \end{array}
          \right),$$
then $$R_{mm}=\langle \widetilde{R}_{mm}\left(
                                          \begin{array}{cc}
                                            z \\
                                            \overline{z}
                                          \end{array}
                                        \right)
, \left(
 \begin{array}{cc}
  z \\
  \overline{z}
  \end{array}
  \right)\rangle.$$ So
$$(1-\Gamma_{K_{m}})R_{mm}\triangleq
\langle (1-\Gamma_{K_{m}})\widetilde{R}_{mm}
\left(
          \begin{array}{c}
   z \\
\overline{z} \\
     \end{array}
      \right),
\left(
          \begin{array}{c}
   z \\
\overline{z} \\
     \end{array}
      \right)
\rangle .$$
By the definition of truncation operator $\Gamma_{K_{m}},$
$$(1-\Gamma_{K_{m}})\widetilde{R}_{mm}=\sum_{|k|> K_{m}}\widehat{\widetilde{R}}_{mm}(k)e^{i \langle k, \theta\rangle},
\;\theta\in \mathbb{T}^{n}_{s_{m}},\;\tau\in \Pi_{m}.$$
Since $\widetilde{R}_{mm}=\widetilde{R}_{mm}(\theta)$ is analytic in $\theta\in \mathbb{T}^{n}_{s_{m}},$
\begin{eqnarray*}
&&\sup_{(\theta, \tau)\in\mathbb{T}^{n}_{s_{m+1}}\times \Pi_{m+1}}\|J(1-\Gamma_{K_{m}})\widetilde{R}_{mm}J\|_{l_{N}\to l_{N}}^{2}
\leq\sum_{|k|> K_{m}}\|J\widehat{\widetilde{R}}_{mm}(k)J\|_{N}^{2}e^{2|k|s_{m+1}}\\
&&\leq\|J\widetilde{R}_{mm}J\|_{\mathbb{T}_{s_{m}}^{n}\times \Pi_{m}}^{2}\sum_{|k|>K_{m}}e^{-2(s_{m}-s_{m+1})|k|}
\\
&&\leq{C^{2}(m)\varepsilon_{m}^{-1}e^{-2K_{m}(s_{m}-s_{m+1})}}\;\mbox{(by \eqref{eq30})}\\
&&\leq C^{2}(m)\varepsilon^{2}_{m}.
\end{eqnarray*}
That is,
$$\|J(1-\Gamma_{K_{m}})\widetilde{R}_{mm}J\|_{\mathbb{T}^{n}_{s_{m+1}}\times \Pi_{m+1}}\leq \varepsilon_{m}C(m).$$
Thus,
\begin{equation}\label{eq9.1}
\|\varepsilon_{m}J(1-\Gamma_{K_{m}})\widetilde{R}_{mm}J\|_{\mathbb{T}^{n}_{s_{m+1}}\times \Pi_{m+1}}
\leq \varepsilon^{2}_{m}C(m)\leq \varepsilon_{m+1}C(m+1).
\end{equation}
  \item Estimate of \eqref{eq5.28}.

Let
\begin{equation*}
S_{m}=\left(
            \begin{array}{cc}
             F^{z{z}}(\theta, \tau) & \frac{1}{2}F^{{z}\,\overline{z}}(\theta, \tau) \\
            \frac{1}{2} F^{z\overline z}(\theta, \tau) & F^{\overline {z}\overline{z}}(\theta, \tau)\\
            \end{array}
          \right),
\;\;\mathcal{J}=\left(\begin{array}{cc}
             0& -\sqrt{-1}id \\
            \sqrt{-1}id & 0\\
            \end{array}
          \right).
\end{equation*}
Then we can write
$$F=\langle S_{m}(\theta)\left(
                       \begin{array}{c}
                         z \\
                         \overline{z} \\
                       \end{array}
                     \right),
\left(
                       \begin{array}{c}
                         z \\
                         \overline{z} \\
                       \end{array}
                     \right)
\rangle =\langle S_{m}u, u\rangle,\;u=\left(
                       \begin{array}{c}
                         z \\
                         \overline{z} \\
                       \end{array}
                     \right).$$
Then
\begin{equation}\label{eq9.2}
\varepsilon_{m}^{2}\{R_{mm}, F\}
=4\varepsilon_{m}^{2} \langle {\widetilde{R}}_{mm}(\theta)\mathcal{J}S_{m}(\theta)u, u\rangle.
\end{equation}
Noting $\mathbb{T}^{n}_{s_{m}}\times \Pi_{m}\supset \mathbb{T}^{n}_{s_{m+1}}\times \Pi_{m+1}.$ By \eqref{eq31} with $l=m, v=m,$
\begin{equation}\label{eq9.02}
\|\widetilde{R}_{mm}(\theta)\|_{\mathbb{T}^{n}_{s_{m+1}}\times \Pi_{m+1}}
\leq \|\widetilde{R}_{mm}(\theta)\|_{\mathbb{T}^{n}_{s_{m}}\times \Pi_{m}}\leq C(m),
\end{equation}
\begin{equation}\label{eq9.002}
\|\widetilde{R}_{mm}(\theta)\|^{\mathcal{L}}_{\mathbb{T}^{n}_{s_{m+1}}\times \Pi_{m+1}}\leq C(m).
\end{equation}
Let $\widetilde{S}_{m}(\theta)=\mathcal{J}S_{m}(\theta).$
Then by Lemmas \ref{lem7.1}, \ref{lem7.4} and \ref{lem7.5} in Section 7, we have
\begin{equation}\label{eq8.04}
\|J \widetilde{S}_{m}(\theta)J\|_{\mathbb{T}^{n}_{s_{m+1}}\times \Pi_{m+1}}\leq C(m+1)\varepsilon_{m}^{-\frac{6(n+1)}{N}},
\end{equation}
\begin{equation}\label{eq8.05}
\|J \widetilde{S}_{m}(\theta)J\|^{\mathcal{L}}_{\mathbb{T}^{n}_{s_{m+1}}\times \Pi_{m+1}}\leq C(m+1)\varepsilon_{m}^{-\frac{12(n+1)}{N}},
\end{equation}
and
\begin{equation}\label{eq9.3}
\|\widetilde{R}_{mm}\mathcal{J}S_{m}\|_{\mathbb{T}^{n}_{s_{m+1}}\times \Pi_{m+1}}=\|\widetilde{R}_{mm}\widetilde{S}_{m}\|_{\mathbb{T}^{n}_{s_{m+1}}\times \Pi_{m+1}}\leq C(m)C(m+1)\varepsilon_{m}^{-\frac{6(n+1)}{N}}.
\end{equation}
Note that the vector field is linear. So, by Taylor formula, one has
$$
\eqref{eq5.28}=\varepsilon_{m}^{2}\langle \widetilde{R}^{*}_{m}(\theta)u, u\rangle, $$
where
$$\widetilde{R}^{*}_{m}(\theta)=\sum_{j=1}^{\infty}\frac{4^{j}\varepsilon_{m}^{j-1}}{j!}\cdot\widetilde{R}_{mm}\underbrace{\widetilde{S}_{m}\cdots \widetilde{S}_{m}}_{j-fold}.$$
By \eqref{eq9.02} and \eqref{eq8.04},
\begin{eqnarray*}
\|J\widetilde{R}^{*}_{m}(\theta)J\|_{\mathbb{T}^{n}_{s_{m+1}}\times \Pi_{m+1}}
&\leq &\sum_{j=1}^{\infty}\frac{C(m)C(m+1)\varepsilon_{m}^{j-1}(\varepsilon_{m}^{-\frac{6(n+1)}{N}})^{j}}{j!}\\
&\leq &C(m)C(m+1)\varepsilon_{m}^{-\frac{6(n+1)}{N}}.
\end{eqnarray*}
By \eqref{eq9.002} and \eqref{eq8.05},
$$\|J\widetilde{R}^{*}_{m}(\theta)J\|^{\mathcal{L}}_{\mathbb{T}^{n}_{s_{m+1}}\times \Pi_{m+1}}\leq C(m)C(m+1)\varepsilon_{m}^{-\frac{12(n+1)}{N}}.$$
Thus,
\begin{equation}\label{eq*1}
||\varepsilon_{m}^{2}J\widetilde{R}_{m}^{*}J||_{\mathbb{T}^{n}_{s_{m+1}}\times \Pi_{m+1}}
\leq C(m)C(m+1)\varepsilon_{m}^{2-\frac{6(n+1)}{N}}\leq C(m+1)\varepsilon_{m+1},
\end{equation}
and
\begin{equation}\label{eq*2}
||\varepsilon_{m}^{2}J\widetilde{R}_{m}^{*}J||^{\mathcal{L}}_{\mathbb{T}^{n}_{s_{m+1}}\times \Pi_{m+1}}
\leq C(m)C(m+1)\varepsilon_{m}^{2-\frac{12(n+1)}{N}}\leq C(m+1)\varepsilon_{m+1}.
\end{equation}

   \item Estimate of \eqref{eq5.27}

By \eqref{eq43},
\begin{equation}\label{eq9.7}
\{N_{m}, F\}=\langle [R^{z \overline{z}}_{mm}]z, \overline{z}\rangle-\Gamma_{K_{m}}R_{mm}-\omega\cdot \partial_{\theta}F\triangleq R_{mm}^{*}.
\end{equation}
Thus,
\begin{equation}\label{eq9.8}
\eqref{eq5.27}=\varepsilon_{m}^{2}\int_{0}^{1}(1-\tau)\{R_{mm}^{*}, F\}\circ X_{\varepsilon_{m}F}^{\tau} d\tau.
\end{equation}
Note $R_{mm}^{*}$ is a quadratic polynomial in $z$ and $\overline{z}.$ So we write
\begin{equation}\label{eq9.9}
R^{*}_{mm}=\langle \mathcal{R}_{m}(\theta, \tau)u, u\rangle,\; u=\left(
                                                                 \begin{array}{c}
                                                                   z \\
                                                                   \overline{z}\\
                                                                 \end{array}
                                                               \right).
\end{equation}
By \eqref{eq29} and \eqref{eq029} with $l=\nu=m,$ and using \eqref{eq8.04} and \eqref{eq8.05},
\begin{equation}\label{eq9.10}
\|J\mathcal{R}_{m}J\|_{\mathbb{T}^{n}_{s_{m}+1}\times \Pi_{m+1}}\leq C(m)\varepsilon_{m}^{-\frac{6(n+1)}{N}},\;\;
\|J\mathcal{R}_{m}J\|^{\mathcal{L}}_{\mathbb{T}^{n}_{s_{m}+1}\times \Pi_{m+1}}\leq C(m)\varepsilon_{m}^{-\frac{12(n+1)}{N}},
\end{equation}
where $\|\cdot\|$ is the operator norm in $h_{N}\times h_{N}\rightarrow h_{N}\times h_{N}.$
Recall $F=\langle S_{m}(\theta, \tau)u, u\rangle.$
Set
\begin{equation}\label{eq9.11}
[\mathcal{R}_{m}, \widetilde{S}_{m}]=2\mathcal{R}_{m}\widetilde{S}_{m}=2\mathcal{R}_{m}\mathcal{J}S_{m}.
\end{equation}
Using Taylor formula to \eqref{eq9.8}, we get
\begin{eqnarray*}
\eqref{eq5.27}&=&\frac{\varepsilon_{m}^{2}}{2!}\{\{R^{*}_{mm}, F\}, F\}+\ldots +\frac{\varepsilon_{m}^{j}}{j!}
\underbrace{\{\ldots\{R^{*}_{mm}, F\},\ldots, F\}}_{j-\mbox{fold}}+\ldots\\
&=&\Bigg\langle\left(\sum_{j=2}^{\infty}\frac{\varepsilon^{j}_{m}}{j!}\underbrace{[\ldots [\mathcal{R}_{m}, \widetilde{S}_{m}], \ldots , \widetilde{S}_{m}]}_{j-\mbox{fold}}\right)u, u\Bigg\rangle\\
&\triangleq &\langle \mathcal{R}^{**}(\theta, \tau)u,u\rangle.
\end{eqnarray*}
By \eqref{eq8.04},\eqref{eq9.10} and \eqref{eq9.11}, we have
\begin{eqnarray}\label{eq9.12}
\|J\mathcal{R}^{**}(\theta, \tau)J\|_{\mathbb{T}^{n}_{s_{m+1}}\times \Pi_{m+1}}
&\leq & \sum_{j=2}^{\infty}\frac{1}{j!}\|J\mathcal{R}_{m}(\theta, \tau)J\|_{\mathbb{T}^{n}_{s_{m}}\times \Pi_{m}}
(2\|J\widetilde{S}_{m}J\|_{\mathbb{T}^{n}_{s_{m+1}}\times \Pi_{m+1}}\varepsilon_{m})^{j}\nonumber\\
&\leq & \sum_{j=2}^{\infty}\frac{C(m)}{j!}\left(\varepsilon_{m}C(m+1)\varepsilon_{m}^{-\frac{6(n+1)}{N}}\right)^{j}\nonumber\\
&\leq & C(m+1)\varepsilon_{m}^{4/3}=C(m+1)\varepsilon_{m+1}.
\end{eqnarray}
Similarly,
\begin{equation}\label{eq9.13}
\|\mathcal{R}^{**}\|^{\mathcal{L}}_{\mathbb{T}^{n}_{s_{m+1}}\times \Pi_{m+1}}\leq C(m+1)\varepsilon_{m+1}.
\end{equation}
  \item Estimate of \eqref{eq5.25}
\begin{eqnarray}\label{eq9.19}
\eqref{eq5.25}=\sum_{l=m+1}^{\infty}\varepsilon_{l}(R_{lm}\circ X_{\varepsilon_{m}F}^{1}).
\end{eqnarray}
Write
$$R_{lm}=\langle \widetilde{R}_{lm}(\theta)u, u\rangle.$$
Then, by Taylor formula:
$$R_{lm}\circ X_{\varepsilon_{m}F}^{1}=R_{lm}+\sum_{j=1}^{\infty}\frac{1}{j!}\langle \widetilde{R}_{lmj} u, u\rangle,$$
where
$$\widetilde{R}_{lmj}=4^{j}\widetilde{R}_{lm}(\theta)\underbrace{\widetilde{S}_{m}(\theta)\cdots \widetilde{S}_{m}(\theta)}_{j-fold}\varepsilon_{m}^{j}.$$
By \eqref{eq30}, \eqref{eq31},
$$\|J\widetilde{R}_{lm}J\|_{\mathbb{T}^{n}_{s_{l}}\times \Pi_{m}}\leq C(l),\;\;\|J\widetilde{R}_{lm}J\|^{\mathcal{L}}_{\mathbb{T}^{n}_{s_{l}}\times \Pi_{m}}\leq C(l).$$
Combing the last inequalities with \eqref{eq8.3} and \eqref{eq8.4}, one has
\begin{eqnarray*}
&&\|J\widetilde{R}_{lmj}J\|_{\mathbb{T}^{n}_{s_{l}}\times \Pi_{m+1}}\\
&\leq & \|J\widetilde{R}_{lm}J\|_{\mathbb{T}^{n}_{s_{l}}\times \Pi_{m+1}}\cdot (||J\widetilde{S}_{m}J||_{\mathbb{T}^{n}_{m+1}\times \Pi_{m+1}}4\varepsilon_{m})^{j}\\
&\leq & C^{2}(m)(\varepsilon_{m}\varepsilon_{m}^{-\frac{6(n+1)}{N}})^{j},
\end{eqnarray*}
where we used $||J^{-1}||_{\mathbb{T}^{n}_{sl}\times \Pi_{m+1}}\leq C,$
and
\begin{eqnarray*}
&&||J\widetilde{R}_{lmj}J||^{\mathcal{L}}_{\mathbb{T}^{n}_{s_{l}}\times \Pi_{m+1}}\\
&\leq & ||J\widetilde{R}_{lm}J||^{\mathcal{L}}_{\mathbb{T}^{n}_{s_{l}}\times \Pi_{m+1}}(||J\widetilde{S}_{m}J||_{\mathbb{T}^{n}_{s_{l}}\times \Pi_{m+1}}4\varepsilon_{m})^{j}\\
&&+
||J\widetilde{R}_{lm}J||_{\mathbb{T}^{n}_{s_{l}}\times \Pi_{m+1}}(||J\widetilde{S}_{m}J||^{\mathcal{L}}_{\mathbb{T}^{n}_{s_{l}}\times \Pi_{m+1}}\varepsilon_{m})^{j}\\
&\leq &C^{2}(m)(\varepsilon_{m}\varepsilon_{m}^{-\frac{12(n+1)}{N}})^{j}.
\end{eqnarray*}
Thus, let
$$\overline{R}_{l,m+1}:=R_{lm}+\sum_{j=1}^{\infty}\frac{1}{j!}\widetilde{R}_{lmj},$$
then
\begin{equation}\label{eq*3}
\eqref{eq5.25}=\sum_{l=m+1}^{\infty}\varepsilon_{l}\langle \overline{R}_{l,m+1}u, u\rangle
\end{equation}
and
\begin{equation}\label{eq*4}
||J\overline{R}_{l,m+1}J||_{\mathbb{T}^{n}_{s_{l}}\times \Pi_{m+1}}\leq C^{2}(m)\leq C(m+1),\;\;
||J\overline{R}_{l,m+1}J||^{\mathcal{L}}_{\mathbb{T}^{n}_{s_{l}}\times \Pi_{m+1}}\leq C^{2}(m)\leq C(m+1).
\end{equation}
As a whole, the remainder $R_{m+1}$ can be written  as
$$C_{m+1}R_{m+1}=\sum_{l=m+1}^{\infty}\varepsilon_{l}(\langle R^{zz}_{l, \nu}(\theta)z,z\rangle
+\langle R^{z\overline{z}}_{l, \nu}(\theta)z,\overline{z}\rangle)+\langle R^{\overline{z}\overline{z}}_{l, \nu}(\theta)\overline{z},\overline{z}\rangle),\;\;\nu=m+1,$$
where, for $p,q\in \{z,\overline{z}\},$ $R_{l\nu}^{p,q}$ satisfies \eqref{eq30} and \eqref{eq31} with
$\nu=m+1, l\geq m+1.$
This shows that the Assumption $(A2)_{\nu}$ with $\nu=m+1$ holds true.

By \eqref{eq5.23},
$$\mu_{j}^{(m)}=\widehat{R}_{mmjj}^{z\overline{z}}(0).$$
In \eqref{eq30} and \eqref{eq31}, taking $p=z, q=\overline{z},$ we have
\begin{eqnarray*}
&&|\mu_{j}^{(m)}|_{\Pi_{m}}\leq
|R^{z\overline{z}}_{mmjj}(\theta, \tau)|/j\leq C(m)/j,\\
&&|\mu_{j}^{(m)}|_{\Pi_{m}}^{\mathcal{L}}\leq |\partial_{\tau}R^{z\overline{z}}_{mmjj}(\theta, \tau)|/j\leq C(m)/j.
\end{eqnarray*}
This shows that the Assumption $(A1)_{\nu}$ with $\nu=m+1$ holds true.
\end{itemize}
\section{Estimate of measure}
In this section, $C$ denotes a universal constant, which may be different in different places.
Now let us return to \eqref{eq7.5}
\begin{equation}\label{eq10.01}
Q_{kij}^{(m)}\triangleq \left\{\tau\in\Pi_{m}\bigg| \big|-\langle k, \omega_{0}\rangle\tau
+\lambda_{i}^{(m)}-\lambda_{j}^{(m)}\big|<\frac{(|i-j|+1)\gamma_{m}}{A_{k}}\right\}.
\end{equation}
Then $i=j$, one has $k\neq 0.$ At this time,
\begin{equation}\label{eq10.02}
Q_{kii}^{(m)}= \left\{\tau\in\Pi_{m}\bigg| \big|\langle k, \omega_{0}\rangle\tau
\big|<\frac{\gamma_{m}}{A_{k}}\right\}.
\end{equation}
It follows
$$ \big|\langle k, \omega_{0}\rangle
\big|<\frac{\gamma_{m}}{(|k|^{2n+3}+8)\tau}.$$
Recall $|\langle k, \omega_{0}\rangle |>\frac{\gamma}{|k|^{n+1}}$.
Then
\begin{equation}\label{eq10.03}
mes Q_{kii}^{(m)}=0.
\end{equation}
Suppose $i\neq j,$ in the following.
If $ Q_{kij}^{(m)}=\varnothing,$ then $mes Q_{kij}^{(m)}=0.$ So we assume $ Q_{kij}^{(m)}\neq\varnothing.$
Then $\exists \,\tau\in \Pi_{m}$ such that
\begin{equation}\label{eq10.04}
|-\langle k, \omega_{0}\rangle \tau+\lambda_{i}^{(m)}-\lambda_{j}^{(m)}|<\frac{|i-j|+1}{A_{k}}\gamma_{m}.
\end{equation}
It follows from \eqref{eq28} and \eqref{eq29} that
\begin{equation}\label{eq10.05}
\lambda_{i}^{(m)}-\lambda_{j}^{(m)}=i-j+O (\frac{\varepsilon_{0}}{i})+O (\frac{\varepsilon_{0}}{j}).
\end{equation}
Moreover,
\begin{equation}\label{eq10.06}
|\lambda_{i}^{(m)}-\lambda_{j}^{(m)}|\geq \frac{1}{2}|i-j|.
\end{equation}
By \eqref{eq10.04} and \eqref{eq10.06}, one has
\begin{eqnarray*}
|\langle k, \omega_{0}\rangle \tau|&\geq &|\lambda_{i}^{(m)}-\lambda_{j}^{(m)}|-\frac{|i-j|+1}{A_{k}}\\
&\geq & \frac{1}{2}|i-j|-\frac{|i-j|+1}{A_{k}}\\
&\geq & \frac{1}{4}|i-j|.
\end{eqnarray*}
Recall $\omega=\omega_{0}\tau.$ So
\begin{equation}\label{eq10.07}
4|\langle k, \omega\rangle|\geq |i-j|.
\end{equation}
Again by \eqref{eq10.04} and \eqref{eq10.05}, we have that, when $\tau\in \Pi_{m}$ such that
\eqref{eq10.04} holds true, the following inequality holds true:
\begin{eqnarray}\label{eq10.08}
|-\langle k, \omega\rangle +i-j|&\leq &\frac{|i-j|+1}{A_{k}}\gamma_{_{m}}+\frac{C_{1}\varepsilon_{0}}{i}+
\frac{C_{2}\varepsilon_{0}}{j}\nonumber\\
&\leq & \frac{|i-j|+1}{A_{k}}\gamma_{_{m}}+\frac{C_{1}\varepsilon_{0}}{i_{0}}+
\frac{C_{2}\varepsilon_{0}}{j_{0}},\;\mbox{if}\;i\geq i_{0},\;j\geq j_{0},
\end{eqnarray}
where $C_{1}>0, C_{2}>0$ are constants.

Thus
\begin{equation}\label{eq10.09}
Q^{(m)}_{kij}\subset \left\{\tau\in\Pi_{m}\big||-\langle k, \omega\rangle +l|<\frac{|l|+1}{A_{k}}\gamma_{m}
+\frac{C_{1}\varepsilon_{0}}{i_{0}}+\frac{C_{2}\varepsilon_{0}}{j_{0}}\right\}\triangleq \widetilde{Q}_{kl},
\end{equation}
when $i\geq i_{0},$ $j\geq j_{0}.$
By \eqref{eq10.07}, one has
\begin{equation}\label{eq10.010}
|l|\leq 4 |\langle k, \omega\rangle|\leq C |k|.
\end{equation}
Note that
$$-\langle k, \omega\rangle +l=-\langle k, \omega_{0}\rangle\tau +l=\tau(-\langle k, \omega_{0}\rangle+\frac{l}{\tau})$$
and
\begin{equation}\label{eq*}
\big|\frac{d}{d \tau}(-\langle k, \omega_{0}\rangle+\frac{l}{\tau})\big|=\frac{|l|}{\tau^{2}}\geq \frac{1}{4}|l|.
\end{equation}
It follows that
\begin{equation}\label{eq10.011}
mes \widetilde{Q}_{kl}\leq \frac{8}{|l|}\left(\frac{|l|+1}{A_{k}}\gamma_{m}+\frac{C_{1}\varepsilon_{0}}{i_{0}}
+\frac{C_{2}\varepsilon_{0}}{j_{0}}\right).
\end{equation}
Take
\begin{equation}\label{eq10.012}
j_{0}=i_{0}=|k|^{n+1}\gamma_{m}^{-1/3}.
\end{equation}
Then
\begin{eqnarray*}
mes \bigcup_{1\leq l\leq C|k|} \widetilde{Q}_{kl}
&\leq & \frac{C|k|\gamma_{m}}{A_{k}}+C\sum_{1\leq |l|\leq C|k|}\frac{1}{|l|}
(\frac{C_{1}\varepsilon_{0}}{i_{0}}+\frac{C_{2}\varepsilon_{0}}{j_{0}})\\
&\leq & \frac{C|k|\gamma_{m}}{A_{k}}+C\gamma_{m}^{1/3}\varepsilon_{0}\frac{\log |k|}{|k|^{n+1}}\\
&\leq & C\gamma_{m}^{1/3}\varepsilon_{0}\frac{\log |k|}{|k|^{n+1}}.
\end{eqnarray*}
Thus,
\begin{equation}\label{eq10.013}
mes \bigcup_{\begin{array}{c}
                                i\geq i_{0} \\
                                j\geq j_{0} \\
                                |i-j|\leq C|k|
                              \end{array}}
                              \!\!\!Q_{kij}^{(m)}\leq C\gamma_{m}^{1/3}\varepsilon_{0}\frac{\log |k|}{|k|^{n+1}}.
 \end{equation}
Now assume
\begin{eqnarray}\label{eq10.014}
i\leq i_{0} \;\;\mbox{or}\;\;j\leq j_{0}\;\;\mbox{and}\;\; |i-j|\leq C |k|.
\end{eqnarray}
By \eqref{eq10.04} and \eqref{eq10.06},  we have
$$\Big|\frac{d}{d\tau}\left(\frac{-\langle k, \omega_{0}\rangle \tau +\lambda_{i}^{(m)}-\lambda_{j}^{(m)}}{\tau}\right)\Big|
\geq \frac{|\lambda_{i}^{(m)}-\lambda_{j}^{(m)}|}{4}\geq \frac{|i-j|+1}{16},$$

\begin{eqnarray}\label{eq10.015}
mes \!\! \!\!\!\bigcup_{\begin{array}{c}
                                i\leq i_{0} \\
                                |i-j|\leq C|k|
                              \end{array}}
                             \!\! \!\!\!Q_{kij}^{(m)}
                             &\leq &
\sum_{\begin{array}{c}
                                1\leq i\leq i_{0} \nonumber\\
                                |i-j|\leq C|k|
                              \end{array}}\!\!\!\frac{2(|i-j|+1)\gamma_{m}}{A_{k}}\cdot \frac{16}{|i-j|+1}\\
                              & \leq &Ci_{0}\frac{C|k|\gamma_{m}}{A_{k}}
                              \leq C|k|^{n+2}\gamma_{m}^{2/3}\frac{1}{A_{k}}\leq \frac{C\gamma_{m}^{2/3}}{|k|^{n+1}},
\end{eqnarray}
and
\begin{eqnarray}\label{eq10.017}
mes \!\! \!\!\!\bigcup_{\begin{array}{c}
                                j\leq j_{0} \\
                                |i-j|\leq C|k|
                              \end{array}}
                             \!\! \!\!\!Q_{kij}^{(m)}
                              \leq \frac{C\gamma_{m}^{2/3}}{|k|^{n+1}}.
\end{eqnarray}
Combining \eqref{eq10.03}, \eqref{eq10.013} \eqref{eq10.015} and \eqref{eq10.017}, we have
\begin{eqnarray}\label{eq10.016}
mes \bigcup_{|k|\leq K_{m}}\bigcup_{i=1}^{\infty}\bigcup_{j=1}^{\infty}Q^{(m)}_{kij}
\leq C\gamma_{m}^{1/3}.
 \end{eqnarray}
 Let
 $$\Pi^{+-}_{m+1}=\Pi_{m}\backslash
 \bigcup_{|k|\leq K_{m}}\bigcup_{i,j=1}^{\infty}Q^{(m)}_{kij}.$$
 Then we have proved the following Lemma \ref{lem10.1}.

\begin{lemma}\label{lem10.1}
$$mes\Pi^{+-}_{m+1}\geq mes \Pi_{m}-C\gamma_{m}^{1/3}.$$
\end{lemma}
\section{Proof of Theorems}
Let
$$\Pi_{\infty}=\bigcap_{m=1}^{\infty} \Pi_{m},$$
and $$\Psi_{\infty}=\lim_{m\rightarrow \infty}\Psi_{0}\circ\Psi_{1}\circ \cdots \circ \Psi_{m}.$$
By \eqref{eq33} and \eqref{eq33H}, one has
$$\Psi_{\infty}: \mathbb{T}^{n}\times \Pi_{\infty}\rightarrow \mathbb{T}^{n}\times \Pi_{\infty},$$
$$||\Psi_{\infty}-id||\leq \varepsilon_{0}^{1/2},$$
and, by \eqref{eq34},
$$H_{\infty}=H \circ \Psi_{\infty}=\sum_{j=1}^{\infty}\lambda_{j}^{\infty}Z_{j}\overline{Z}_{j},$$
where
$$\lambda_{j}^{\infty}=\lim_{m\rightarrow \infty}\lambda_{j}^{(m)}.$$
By \eqref{eq28} and \eqref{eq29}, the limit $\lambda_{j}^{\infty}$ does exists and
$$\lambda_{j}^{\infty}=j+O (\frac{\varepsilon_{0}}{j}):=\sqrt{j^2+\xi_j}.$$
Introduce a transformation $Z=(Z_j\in\mathbb C:\, j\ge 1)\mapsto \, v(t,x)$ by
\[v(t,x)=\sum_{j=1}^\infty q_j(t)\, \sin\, j\, x.\]
\begin{equation}\label{eq18}
Z_{j}=\frac{\sqrt{-1}}{\sqrt{2}}(q_{j}+\sqrt{-1}p_{j}),\;\;
\overline{Z}_{j}=\frac{\sqrt{-1}}{\sqrt{2}}(q_{j}-\sqrt{-1}p_{j}).
\end{equation}
Let
\[\Phi=(\mathcal{S}\mathcal{T}\mathcal{G}\Psi_{\infty}\mathcal{G}^{-1}\mathcal{T}^{-1}\mathcal{S}^{-1})^{-1}.\]
Then $\Phi$ is a symplectic transformation and changes \eqref{eq1} subject to \eqref{eq2}  into
\eqref{yuan1.7}.
Also, the transformation $\Phi$ change the wave operator
\[\mathcal L_V:\; \,\mathcal L_V u(t,x)=(\partial_t^2-\partial_x^2+\varepsilon V(\omega\, t,x))\, u(t,x),\quad u(t,-\pi)=u(t,\pi)=0 \] into
\[\mathcal L_M:\; \, \mathcal L_M v(t,x)=(\partial_t^2-\partial_x^2+\varepsilon M_\xi)\, v(t,x),\quad v(t,-\pi)=v(t,\pi)=0, \]
which possesses the property of pure point spectra and zero Lyapunov exponent.

This completes the proof of Theorem \ref{thm1.1}.
\section*{Acknowledgements}
The work was supported in part by National
Nature Science Foundation of China (11601277).

\end{document}